\documentclass[a4paper,11pt,english]{article}
\usepackage[utf8]{inputenc}
\usepackage[T1]{fontenc}
\usepackage{babel}

\usepackage{fullpage}
\usepackage{units}
\usepackage{enumerate,enumitem}
\usepackage{adjustbox}

\usepackage{array}
\usepackage{float}
\usepackage{makecell}

\usepackage{dsfont,amsmath,amsthm,amsfonts,amssymb,upgreek,eufrak}
\usepackage{caption}

\usepackage{graphicx}
\usepackage{multirow}
\usepackage{color}
\usepackage[all]{xy}
\usepackage{stmaryrd}

\usepackage[normalem]{ulem}

\usepackage{cancel}

\usepackage{url}
\usepackage[nospace]{cite}
\usepackage[pdftex,colorlinks=true,pagebackref=true
,linkcolor=blue,bookmarks=true,bookmarksopen=true,citecolor=blue]{hyperref}
\usepackage[capitalize]{cleveref} 

\makeatletter
\newcommand{\xRightarrow}[2][]{\ext@arrow 0359\Rightarrowfill@{#1}{#2}}
\makeatother

\usepackage{tikz}
\usetikzlibrary{arrows}
\usepackage{pgf, tikz}
\usetikzlibrary{arrows, automata}

\usepackage{xcolor}
\usepackage{graphicx}

\numberwithin{equation}{section} 

\newtheorem{thm}{Theorem}[section]
\newtheorem{defi}[thm]{Definition}
\newtheorem{rem}[thm]{Remark}
\newtheorem{prop}[thm]{Proposition}
\newtheorem{lem}[thm]{Lemma}

\newtheorem*{hypothesis}{Hypothesis}

\newcommand{\eps}{\varepsilon}

\newcommand{\bbZ}{\mathbb{Z}}

\newcommand{\env}{\mathbf{w}}
\newcommand{\ww}{\mathrm{w}}

\usepackage{color}
\usepackage{cite}

\definecolor{ao}{rgb}{0.0, 0.5, 0.0}

\usepackage{bigints}

\usepackage[blocks]{authblk}

\author[1]{Thibaut Duboux}
\author[2]{Lucas Gerin}
\author[1]{Yoann Offret}

\affil[1]{Institut de Mathématiques de Bourgogne (IMB) - UMR CNRS 5584\\
	Université Bourgogne Europe, 21000 Dijon, France}

\affil[2]{CMAP, \'Ecole Polytechnique\\
	Route de Saclay, 91120 Palaiseau, France}

\usepackage{todonotes}

\begin{document}
	
\title{\bf\centering \Large Maximal Entropy Random Walks in $\mathbb{Z}$: Random and non-random environments}

\maketitle

\begin{abstract}
	The Maximal Entropy Random Walk (MERW) is a natural process on a finite graph, introduced a few years ago with motivations from theoretical physics. The construction of this process relies on Perron-Frobenius theory for adjacency matrices. 
	
	Generalizing to infinite graphs is rather delicate, and in this article, We study in detail specific models of the MERW on $\mathbb{Z}$ with loops, for both random and non-random loops.  
	Thanks to an explicit combinatorial representation of the corresponding Perron-Frobenius eigenvectors, we are able to precisely determine the asymptotic behavior of these walks. We show, in particular, that essentially all MERWs on $\mathbb{Z}$ with loops have positive speed.\\
	
	
		 \noindent
	{\bf AMS Classification (2020):}  60C05, 60K37, 60G50, 82B41
\end{abstract}

\tableofcontents

\section{Introduction}

\subsection{Maximal Entropy Random Walks on finite and infinite graphs}

Throughout this article, $G$ is a strongly connected weighted directed graph with a countably finite or infinite vertex set $I$. Below, we will mostly consider the case where $I = \mathbb{Z}$ or a finite interval of $\bbZ$. Loops and multiple edges are allowed.  

Let $A = (a_{i,j})_{i,j \in I}$ be the associated weighted adjacency matrix; $a_{i,j}$ is non-negative and represents the weight of the directed edge $i \to j$. In the case where $a_{i,j} \geq 0$ is an integer $k$, we interpret this as $k$ edges with unit weight from $i$ to $j$. When $a_{i,j} = 0$, this indicates that there is no edge from $i$ to $j$.

If for all $i \in I$, 
\begin{equation}\label{ass0}
\sup_{i \in I} \sum_{j \in I} a_{i,j} < \infty,
\end{equation}
the \emph{standard random walk} on $G$ is the Markov chain with transition probabilities given by  
\begin{equation}\label{srw}
p_{i,j} = \frac{a_{i,j}}{\sum_{\ell \in I} a_{i,\ell}}.
\end{equation}
The standard random walk is a fundamental stochastic process in probability, statistical physics, and network analysis. In this article, we consider a somewhat less well-known process: the \emph{Maximal Entropy Random Walk} (MERW). 

Let us introduce this process. If $G$ is finite, its construction relies on Perron-Frobenius theory\footnote{See \cite[Sec. 6.1]{bapat2010graphs} for a background on Perron-Frobenius in the context of graphs.}. The Perron-Frobenius theorem ensures that the spectral radius $\lambda$ of the finite non-negative matrix $A$ is an eigenvalue of $A$ (resp. $A^T$) and that there exists a positive right (resp. left) $\lambda$-eigenvector $\psi = (\psi_i)_{i \in I}$ (resp. $\varphi  = (\varphi_i)_{i \in I}$), unique up to a multiplicative constant (by the connectedness of $G$).

The \emph{Maximal Entropy Random Walk} (MERW) associated with $G$ is the discrete-time Markov chain $(X_n)_{n \geq 0}$ on $G$ whose transition probabilities are given by  
\begin{equation}\label{eq:MERWfinite}
p_{i,j} = a_{i,j}\frac{\psi_j}{\lambda \psi_i}.
\end{equation}
(It is straightforward to check that \eqref{eq:MERWfinite} indeed defines a transition matrix.)


If all the sums $\sum_{j \in I} a_{i,j}$ for $i \in I$ are equal to some $r > 0$, that is, if $G$ is an $r$-regular graph, then $\lambda = r$, and $\psi \equiv 1$ is a positive $\lambda$-eigenfunction. In this case, the MERW coincides with the standard random walk \eqref{srw}. In the general case, these two processes are significantly different.  Note also that if the matrix $A$ is symmetric, then $\varphi = \psi$ (up to a positive constant), and it is easy to see that $\pi = \psi^2$ is a reversible measure.  \label{p:reversible} 

In this setting, MERWs on finite graphs were introduced in \cite{Localization}.  
What makes these processes natural is that, as the name suggests, they maximize the entropy of trajectories (see \cref{prop:MaxEntropy} below for a formal statement). The original motivation for MERWs seems to trace back to the path-integral formalism in Quantum Mechanics (see \cite[Sec.~2]{burda2010various}), and their origins also partly lie in MCMC methods \cite{hetherington1984observations}. More recently, MERWs have been proposed for use in algorithms for detecting communities in complex networks \cite{ochab2013maximal}.

The case where $G$ is infinite is more subtle and has been addressed more recently. In order to generalize \eqref{eq:MERWfinite}, we need the theory of non-negative infinite matrices developed in \cite{vere1967ergodic,VJII}.  Let $a_{i,j}^{(n)}$ be the $(i,j)$-th entry of the $n$th power $A^n$. Then, the quantity  
\begin{equation}\label{spectralradius}
\lambda = \limsup_{n \to \infty} \left(a_{i,j}^{(n)}\right)^{1/n}
\end{equation} 
does not depend on $i, j\in I$ (see \cite[Th.A]{vere1967ergodic}). Moreover, assumption \eqref{ass0} guarantees that $0 < \lambda < \infty$.  The value of $\lambda$ is sometimes referred to as the \emph{combinatorial spectral radius} since it indicates that the weighted number of trajectories of length $n$ from the vertex $i$ to $j$ grows approximately as $\lambda^n$ as $n$ goes to infinity. It plays the role of the Perron-Frobenius eigenvalue.

Regarding positive solutions of $A\psi = \lambda \psi$ (or $A^T\varphi = \lambda \varphi$), there are mainly two cases to distinguish, depending on whether  
\begin{equation}\label{rrec}
\sum_{n=0}^\infty \frac{a_{i,j}^{(n)}}{\lambda^n} = +\infty \quad\text{($R$-recurrence)}\quad \text{or} \quad \sum_{n=0}^\infty \frac{a_{i,j}^{(n)}}{\lambda^n} < +\infty \quad\text{($R$-transience)}.
\end{equation}

\begin{rem}
	The letter $R$ refers to the \emph{convergence parameter} introduced by Vere-Jones 
	\cite{vere1967ergodic}. It is defined as the radius of convergence of the power series 
	$\sum_{n\ge0} a_{i,j}^{(n)} z^n$, which in our setting corresponds to $R = 1/\lambda$.
\end{rem}

As noted in \cite{vere1967ergodic}, if one of these conditions holds for some $i, j \in I$, then it holds for all $i, j \in I$. In each of these situations, we say that $A$, or equivalently the graph $G$, is $R$-recurrent or $R$-transient.

In the $R$-recurrent case, there exists a unique (up to a multiplicative constant) positive right (resp. left) $\lambda$-eigenfunction $\psi$ (resp. $\varphi$). 

By contrast, in the $R$-transient situation, neither existence nor uniqueness is guaranteed. In that case, the set of positive solutions to $A\psi = \lambda\psi$, normalized such that $\psi_\mathtt{o} = 1$ for some fixed $\mathtt{o} \in I$, is a convex set.  As a consequence, this set can be described by the extremal solutions, similarly to the Martin boundary associated with a transient Markov kernel. 

We refer to \cite{Duboux} for a more comprehensive study of these situations and their consequences for the corresponding MERWs.

\begin{defi}\label{def:MERW_infinite}
	Let $A$ be an infinite weighted adjacency matrix satisfying \eqref{ass0}, and let $\psi$ be a positive solution to $A\psi=\lambda \psi$, where $\lambda$ is defined by \eqref{spectralradius}.  
	
	The {Maximal Entropy Random Walk} (MERW) associated with $\psi$ is the discrete-time Markov chain $(X_n)_{n\geq 0}$ on $G$ whose transition probabilities are given by
	\begin{equation}\label{eq:MERWinfinite}
p_{i,j}=a_{i,j}\frac{\psi_j}{\lambda\,\psi_i}.
	\end{equation} 	
\end{defi}

Equation (\ref{eq:MERWinfinite}) evokes of the well-known Doob $h$-transform, 
which is commonly used when conditioning stochastic processes to remain within a 
specific domain (see \cite{doob}; see also \cite{Duboux} for further connections with MERW).

Let us state some fundamental known features of MERWs. To this end, let $q$ be the kernel of a Markov chain \emph{adapted} to $A$, in the sense that $a_{i,j}=0 \Longleftrightarrow q_{i,j}=0$. If $q$ is the kernel of a positive recurrent Markov chain, with $\pi$ as its invariant probability measure, then the (stationary and asymptotic) entropy rate $\mathcal{E}(q)$ is defined by\footnote{Here we take, as usual, the convention $0 \log(0) = 0$.}  
$$
\mathcal{E}(q) = -\sum_{i,j\in I} \pi_{i}\, q_{i,j}\log \left(\frac{q_{i,j}}{a_{i,j}}\right).
$$

Proposition \ref{prop:MaxEntropy} below shows that, as their name suggests, MERWs indeed maximize entropy. Note that when the graph is finite, this can be obtained by standard optimization methods, noting that $\log(\lambda)$ is an upper bound for $\mathcal{E}(q)$ (see \cite{duda2012extended,dixit2015stationary}). Again, the theory is more subtle for infinite graphs, and the situation was clarified in \cite{Duboux}.

In the sequel, $A$ is said to be positive (resp. null) $R$-recurrent when it is $R$-recurrent and, for some (equivalently all) $i, j \in I$, one has  
$$
\lim_{n \to \infty} \frac{a_{i,j}^{(n)}}{\lambda^n} \neq 0 \quad \left(\text{resp. } \lim_{n \to \infty} \frac{a_{i,j}^{(n)}}{\lambda^n} = 0\right).
$$  

\begin{prop}[MERWs maximize entropy (see \cite{Duboux})]\label{prop:MaxEntropy}
	The supremum of $\mathcal{E}(q)$, among all adapted positive recurrent Markov kernels $q$ on $G$, is equal to $\log(\lambda)$, and one can even restrict the supremum to all adapted positive recurrent Markov chains on finite strongly connected subgraphs $H \subset G$. Furthermore:
	\begin{enumerate}
		\item If $A$ is positive $R$-recurrent, then the MERW is the unique maximizer of $\mathcal{E}(q)$ among all adapted positive recurrent Markov chains on $G$.
		\item If $A$ is null $R$-recurrent, the supremum is not achieved, and any maximizing sequence of positive recurrent Markov kernels converges pointwise to the MERW kernel.
		\item If $A$ is $R$-transient and the graph is locally finite, then any maximizing sequence of positive recurrent kernels is tight, and any limit point is a MERW. However, not all MERWs can necessarily be approximated in this way.
	\end{enumerate}
\end{prop}

Let us mention another important property of MERWs, which is closely related to entropy maximization (see \cite[Sec.3.1.1]{duda2012extended} and \cite[Proposition 2.3]{Dovgal} for more details). 

For every length $\ell \geq 1$ and every $\gamma = ({i_0}, \cdots, {i_\ell})$ in $I$, one has
\begin{equation}\label{UniformPaths}
\mathbb{P}_{{i_0}}(X_1 = {i_1},\cdots, X_\ell = {i_\ell} ) 
\stackrel{\text{\eqref{eq:MERWinfinite}}}{=} \frac{a_{i_0,i_1}}{\lambda} \frac{\psi_{i_1}}{\psi_{i_0}}
 \dots \frac{a_{i_{\ell-1},i_\ell}}{\lambda} \frac{\psi_{i_\ell}}{\psi_{i_{\ell-1}}}=\ a_\gamma\frac{\psi_{i_\ell}}{\lambda^{\ell}\psi_{i_0}}.
\end{equation}
where we set $a_\gamma = a_{i_0,i_1} \cdots a_{i_{\ell-1},i_\ell}$. 

This formula implies that the conditional distribution of the $\ell$-length trajectory of $(X_n)_{n\geq 0}$, given any two extremities ${i_0}, {i_\ell}$, depends only on the product $a_\gamma$ of weights along the trajectory.  
In particular, if all nonzero weights are identical (e.g., when $A$ is a $0/1$-matrix), then every path between a given starting point ${i_0}$ and an ending point ${i_\ell}$ is equally likely.

\subsection{Our results: MERWs on $\bbZ$}

Burda \emph{et al.} \cite{Localization} (see also \cite{duda2012extended}) conducted a thorough study of the MERW on some simple graphs. For example, they proposed studying the MERW on some random perturbations of $(\mathbb{Z}/n\mathbb{Z})^d$ by adding independent random loops. They notably demonstrated numerically that the MERW exhibits a strong tendency to localize on sites with loops (as these are naturally favored by $\psi$), in contrast to the diffusive behavior of the standard random walk.  In light of Duda's experiments, we specifically aim to investigate the localization (or delocalization) of the MERW on a random perturbation of $\mathbb{Z}$ with \emph{i.i.d.\@} loops.

Let us present our model. Let $\env = (\ww_k)_{k \in \mathbb{Z}}$ be a bounded sequence of non-negative numbers, referred to as the \emph{loop environment}. We do not specify yet whether $\env$ is random or deterministic.   Let $\mathbb{Z}^{(\env)}$ be the symmetric graph with vertex set $\mathbb{Z}$, where each pair of neighboring vertices is connected by an edge of weight $1$, and each vertex $i \in \mathbb{Z}$ has a self-loop of weight $\ww_i$.  

The corresponding weighted infinite adjacency matrix is thus:

\begin{minipage}{0.35\textwidth}
	\centering
	\includegraphics[width=\linewidth]{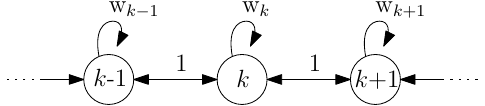} 
\end{minipage}
\hfill
\begin{minipage}{0.65\textwidth}
	\vspace{0pt} 
	
	\begin{equation}\label{eq:A}
	A=(a_{i,j})_{i,j\in\mathbb Z}=
	\begin{matrix}
	\vdots \\
	k-1  \\ 
	k\\ 
	k+1\\ 
	\vdots \\
	\end{matrix}
	\begin{pmatrix} 
	\ddots  & &   &   &  \\ 
	1 & \ww_{k-1} & 1   &   &   \\ 
	& 1 & \ww_k & 1 &  \\ 
	& & 1 & \ww_{k+1} & 1  \\ 
	&   &   & &  \ddots  \\ 
	\end{pmatrix}.
	\end{equation}
	
\end{minipage}
Besides, the corresponding {combinatorial spectral radius} $\lambda$ given in \eqref{spectralradius} is well defined and one can easily check  that 
$$
\lambda\in \left[2,2+\sup_{k\in\mathbb Z} \ww_k\right].
$$
We now introduce the main object of this paper, which is the specialization of \cref{def:MERW_infinite} to the lattice $\mathbb{Z}^{(\env)}$.
\begin{defi} \label{themodel}
	Let $\psi$ be a positive $\lambda$-eigenvector of the matrix ${A}$ defined by \eqref{eq:A}, {i.e.\@} satisfying  
	\begin{equation}\label{harmonic1}
	\forall k\in\mathbb{Z},\quad \psi_{k+1}+\ww_k \psi_k+\psi_{k-1} = \lambda\,\psi_k.
	\end{equation}
	
	The corresponding MERW on $\mathbb{Z}$ with loop environment $\env$, associated with $\psi$, is the Markov chain $(X_n)_{n\geq 0}$ on $\mathbb{Z}$ with transition probabilities given for all $i\in\mathbb{Z}$ by
	\begin{equation}\label{eq:ProbasTransition}
	p_{i,i+1}= \frac{1}{\lambda}\frac{\psi_{i+1}}{\psi_{i}},\qquad
	p_{i,i}= \frac{\ww_i}{\lambda},\qquad
	p_{i,i-1}=
	\frac{1}{\lambda}\frac{\psi_{i-1}}{\psi_{i}} .
	\end{equation}
	We shall denote by $\mathbb{P}^{\env}_k$ the distribution of the process $(X_n)_{n\geq 0}$ starting from $X_0=k$.
\end{defi}  

The distribution $\mathbb{P}^{\env}_k$ is referred to as the \emph{quenched} law of the random walk (in contrast with the \emph{annealed} law, which arises when integrating over a random $\env$). Naturally, the distribution $\mathbb{P}^{\env}_k$ depends on the choice of $\psi$.

\begin{rem}
	Before proceeding, we highlight two  points regarding the loop environment:  
	\begin{itemize}
		\item If $\ww_i=0$ ({i.e.\@} there are no loops) for every $i$, then $\mathbb{P}^{\env}_k$ corresponds to the standard random walk on $\mathbb{Z}$. From now on, we exclude this trivial case.
		\item Throughout this article, we assume that loop environments are bounded. This assumption is crucial not only for the proofs but also for defining the model itself. Indeed, if the $\ww_i$'s were unbounded, then $\lambda=+\infty$, making it unclear how to properly define MERWs.
	\end{itemize}
\end{rem}

An important feature of our model is that $\mathbb{Z}^{(\env)}$ is $R$-transient (in the sense of \eqref{rrec}). Indeed, as we will see below in \cref{solution}, the matrix $A$ has two extremal positive eigenvectors, $\psi^+$ and $\psi^-$. Consequently, there exist infinitely many MERWs, one for each convex combination of $\psi^+$ and $\psi^-$. The two MERWs associated with these extremal eigenvectors, denoted below as $(X_n^+)_{n\geq 0}$ and $(X_n^-)_{n\geq 0}$, play a particular role.\\

\noindent
Our main results  are the following:

\begin{itemize}
	\item In \cref{sec:Deterministic}, we establish several results for a fixed deterministic environment:
	\begin{itemize}
		\item In \cref{solution}, we construct, for every \emph{nice} non-random environment $\env$, two linearly independent positive $\lambda$-eigenvectors of $A$ (see \cref{def:nice} for the definition of a nice environment). In particular, $A$ is $R$-transient. This implies that for every nice environment, the MERW associated with any $\lambda$-eigenvector is transient (\cref{coro:transience_mixture}). 
		\item Additionally, in \cref{Th:psi_combi}, we provide an explicit combinatorial description of the extremal eigenvectors $\psi^+$ and $\psi^-$, valid for every nice environment. We obtain the expressions
$$
	\psi_i^+=\begin{cases}
	\beta_{-1}\dots \beta_{i},&\text{for }i<0,\\[5pt]
	1, &\text{for } i=0,\\[5pt]
	{(\beta_0\beta_1\beta_{2}\dots \beta_{i-1})^{-1}},&\text{for }i>0,
	\end{cases}
	\qquad	
	\psi_i^{-}=\begin{cases}
	{(\alpha_0\alpha_{-1}\dots \alpha_{i+1})^{-1}},&\text{for }i<0,\\[5pt]
	1, &\text{for }i=0,\\[5pt]
	\alpha_1\alpha_{2}\dots \alpha_{i},&\text{for }i>0.
	\end{cases}
	$$
where $(\alpha_i)$'s and $(\beta_i)$'s can be written either as the evaluation of certain generating series (see \cref{eq:def_alpha_beta}) or as explicit continued fractions (see  \cref{eq:alpha_fraction_continue}). This description appears interesting independently of the MERW context.
	\end{itemize}
	\item In \cref{sec:Random}, we consider random loop environments:
	\begin{itemize} 
		\item A consequence of our combinatorial description is that if $(\ww_i)_{i\in\mathbb{Z}}$ are bounded i.i.d.\@ random variables, then both sequences $(\beta_i)_{i}$ and  $(\alpha_{-i})_{i}$ are actually Markov chains. It turns out that the MERW associated to an extremal $\lambda$-eigenvector is a random walk in an ergodic environment.\\
		 This allows us to use existing criteria to prove in \cref{th:transient} that the corresponding MERW is transient with a constant linear speed, both for quenched and annealed distributions (see \cref{fig:simus} for simulations). 
		\item MERWs associated with a generic non-extremal $\lambda$-eigenvector have a more complex structure (they are not random walks in an ergodic environment). However, we are still able to describe their asymptotic behavior using a coupling with the extremal MERWs. This is the purpose of \cref{th:mixture}, which can be seen as the most important result of our article. Let $X_n^{(\kappa)}$ be the MERW corresponding to the eigenvector $\kappa \psi^+ + (1-\kappa)\psi^-$ then, conditionally to $X_n^{(\kappa)}\to +\infty$, 
$$
\lim_{n\to+\infty}\frac{X_n^{(\kappa)}}{n} = v,
$$
for some explicit $v>0$ (see \cref{th:mixture} for the full statement). This shows that all MERWs have a linear speed under the quenched distribution. As a consequence, MERWs on $\bbZ^{(\env)}$ are not localized.
		\item In \cref{sec:Bernoulli} we illustrate a distinctive feature of the model related to its infinite-dimensional nature by carrying out some explicit calculations when the environment is given by i.i.d. Bernoulli random variables.\\
		 We analyze the limiting speed $ v_{p,M}:=\lim_n \tfrac{1}{n}X_n^+ $ where each $\mathbb{P}(\ww_k=M)=p=1-\mathbb{P}(\ww_k=0)$.
Specifically, we show in \cref{asympbernoulli} that $\lim_{p\to 0+} v_{p,M}\neq 0$: the limiting speed of the process abruptly jumps from $0$ to a strictly positive value when i.i.d.\@ Bernoulli random loops are added. The reader is invited to observe  \cref{vitesseM,vitessep}, which illustrate the rather complex behavior of $ v_{p,M}$.
	\end{itemize}
	\item In the Appendix, we provide explicit computations for a deterministic loop environment. Comparing this with the case of Bernoulli random environments reveals the significant impact of randomness in loops.
\end{itemize}

\begin{figure}[H]
\begin{center}
\begin{tabular}{c c}
\includegraphics[width=7.6cm]{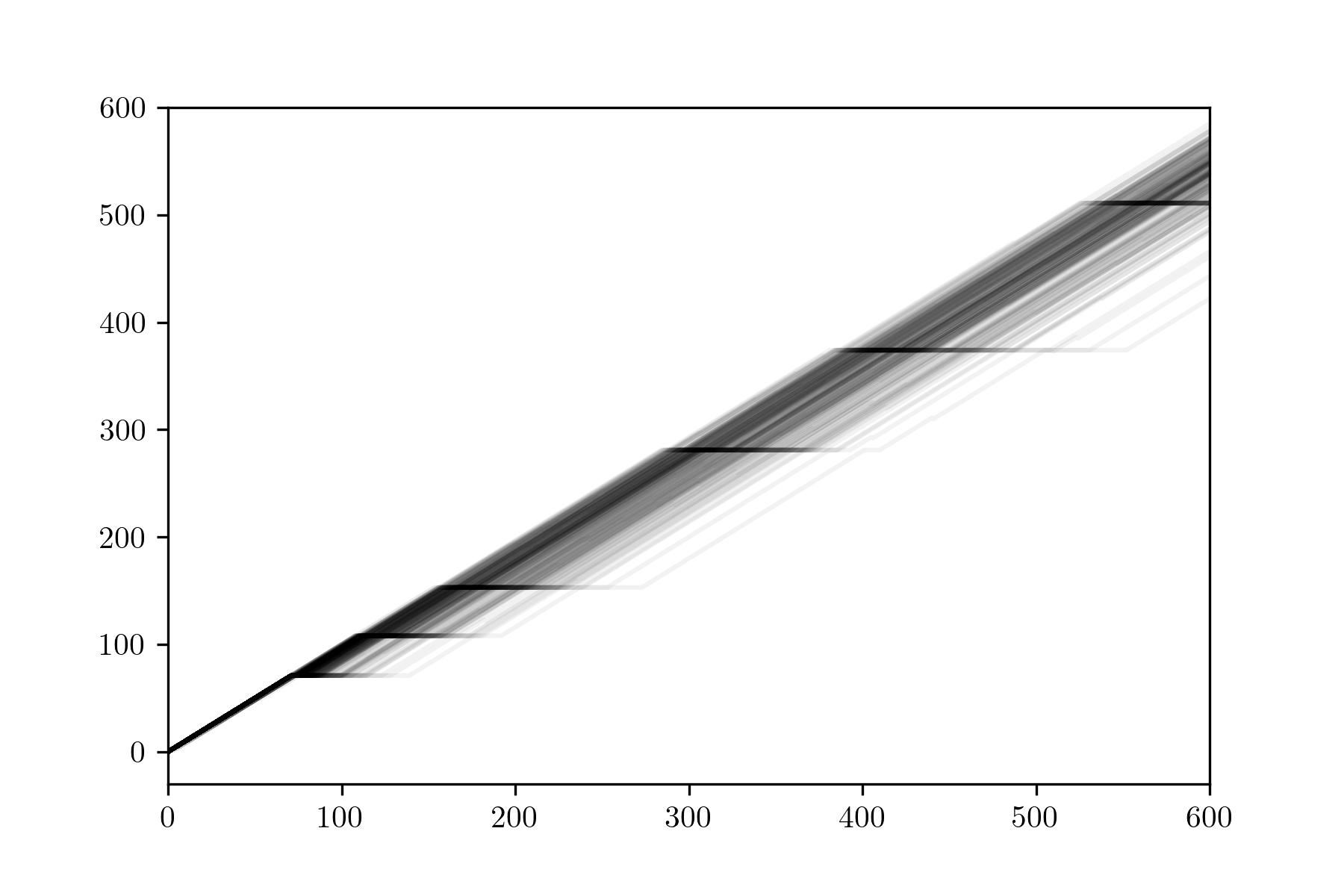} & \includegraphics[width=7.6cm]{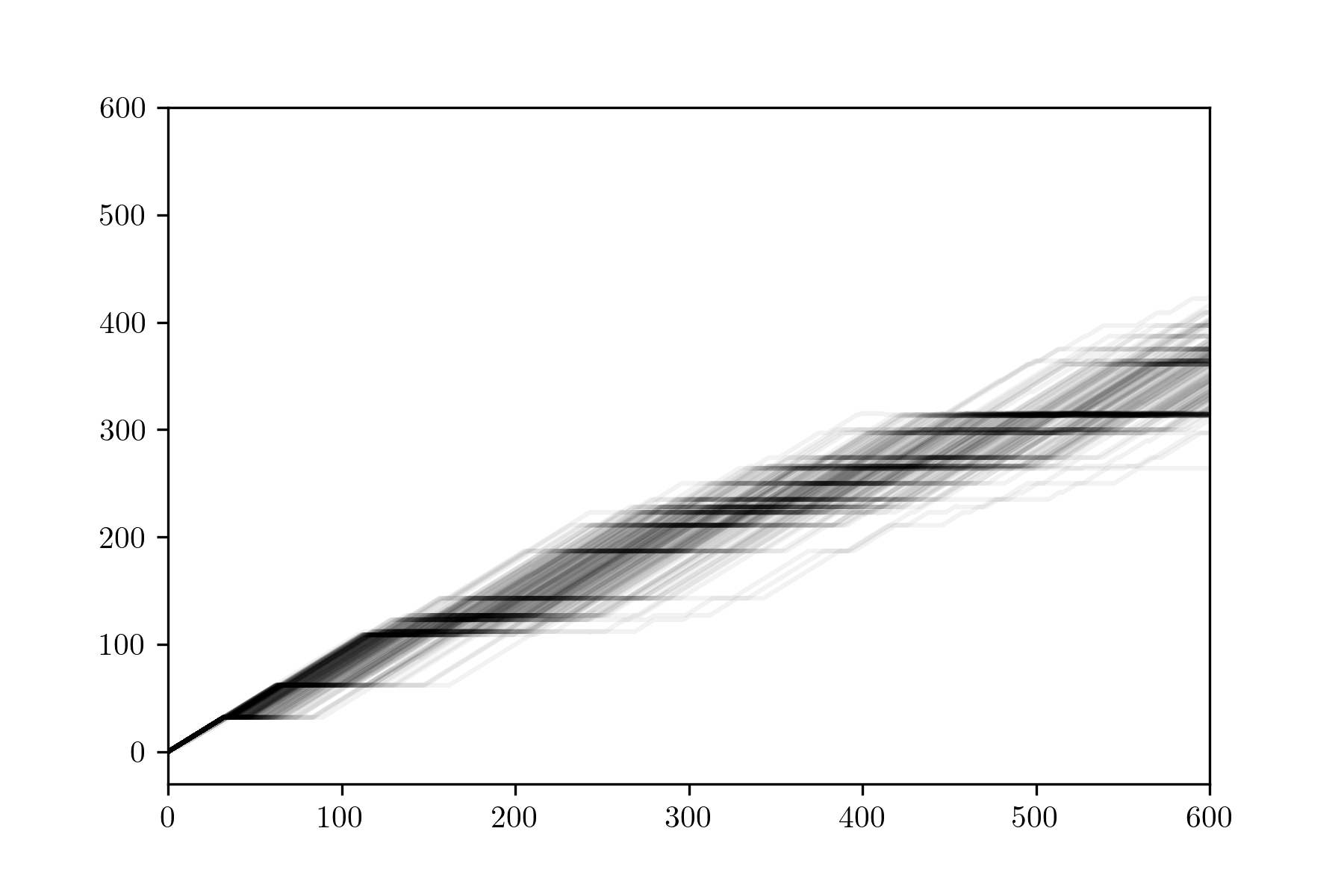} 
\end{tabular}
\caption{Left: $200$ independent simulations of the MERW $(X_n^+)_{n\geq 0}$ up to $n=600$ in the same realization of the random {i.i.d.\@} loop environment. There is a  loop of weight $M=20$ at each vertex with probability $p=0.02$. Right: Same parameters except $p=0.05$.\\
These simulations support 
  \cref{asympbernoulli}: the asymptotic speed $v_{p,M}$ of the MERW is decreasing in $p$.
}
\label{fig:simus}
\end{center}
\end{figure}

\subsubsection*{Connection with the Parabolic Anderson Model}

MERW are closely related to spectral properties of discrete Schrödinger operators with random potentials, central to the study of \emph{Anderson localization}. 

The standard Anderson model considers the operator
\begin{equation*}
\mathcal{H} \psi_n = \psi_{n-1} + \omega_n \psi_n + \psi_{n+1} = \underbrace{\psi_{n-1}-2\psi_n+\psi_{n+1}}_{=\Delta\psi_n} + \underbrace{(2+\omega_n)}_{=\widetilde w_n}\psi_n,
\end{equation*}
where $(\omega_n)_{n \in \mathbb{Z}}$ is an \emph{i.i.d.\@} sequence of random variables, typically taking values in $\mathbb{R}$. 

It is well known that, under mild conditions on $(\omega_n)_{n\geq 0}$, the spectral measure of $\mathcal{H}$ in $\ell^2(\mathbb{Z})$ is almost surely purely ponctual, with exponentially localized eigenfunctions. We refer to \cite{anderson1,carmona1987anderson,kunz1980spectre} for more details.

In contrast, MERWs are associated with the combinatorial spectral radius $\lambda$ of the adjacency matrix $A$ of a graph rather than the spectrum in $\ell^2(\mathbb{Z})$. MERWs  satisfy a discrete Schrödinger-type equation, where $\psi$ is a positive eigenfunction, not necessarily in $\ell^2(\mathbb{Z})$. 

This setting is reminiscent of the study of positive solutions of Schrödinger-type equations in the Parabolic Anderson Model (PAM), which examines
\begin{equation*}
\frac{\partial u}{\partial t} = \Delta u + {\widetilde\omega}\, u,
\end{equation*}
with a random potential $ \widetilde \omega$ (see, \emph{e.g.\@} \cite{anderson2}).

A standard and very powerful tool for the probabilistic analysis of the PAM 
is the Feynman--Kac representation of the solution $u$. It takes the form
\begin{equation}\label{eq:feynman-kac}
u(t,x) = \mathbb{E}_0\!\left[
\exp\!\left( \int_0^t \omega_{X(s)}\,\mathrm{d}s \right)
\mathbf{1}_{\{X(t)=x\}}
\right], 
\end{equation}
where $(X(s))_{s \geq 0}$ denotes a continuous-time random walk on $\mathbb{Z}^d$ 
with generator $\Delta$.

 Although the precise connection between MERW and the PAM remains to be clarified, 
some of the explicit computations carried out here for the MERW, such as the determination 
of the spectral radius and certain localization properties, may provide intuition for 
understanding related phenomena in the PAM. Conversely, techniques developed for the PAM, 
in particular those concerning intermittency, could also offer potential insights for 
further studies on MERW.

%

\section{MERW on $\bbZ$: deterministic loop environments}
\label{sec:Deterministic}
The main goal of this section is to state several results for a fixed loop environment. In particular 
we provide in \cref{Th:psi_combi} an explicit combinatorial description of the extremal eigenvectors $\psi^+$ and $\psi^-$.  We begin with simple properties of the matrix $A$ given by \cref{eq:A}.

\subsection{Preliminaries}

We will mostly consider \emph{nice} environments $\env$, in the following sense. 

\begin{defi}\label{def:nice} Let $M>0$ be fixed. The loop environment $\env=(\ww_{i})_{i\in\bbZ}$ is said to be \emph{$M$-nice} if $\env$ is bounded by $M$, non identically equal to $M$, and if for all $\varepsilon>0$ and every integer $r\geq 0$ there exists $i\in\bbZ$ such that 
	\begin{equation}\label{hypnice}
	\ww_{i},\dots,\ww_{i+r}\text{ are all }\geq M-\eps.
	\end{equation}
\end{defi}

The assumption of an $M$-nice environment is rather restrictive, but it allows for an explicit computation of the spectral radius, see Lemma~\ref{prop:nice} below.In addition, it is well suited to the case of an \emph{i.i.d.\@} environment. Moreover, as mentioned in the introduction, we eliminate the case where $w_i=M$ for every $i$ in order to avoid trivialities: in this case the only $\lambda$-eigenvectors are the constant ones and the MERW coincides with the standard random walk.

The following lemma is not surprising but its proof contains several estimates that will be needed later.
\begin{lem}\label{prop:nice}
Assume that $\env$ is $M$-nice. Then the combinatorial spectral radius of the matrix $A$ defined by \cref{eq:A} is given by $\lambda=2+M.$
\end{lem}

\begin{proof}[Proof of \cref{prop:nice}]
First note that, for every $n\geq 1$ and $i,j\in\mathbb Z$, one has $a^{(n)}_{i,j}\leq (2+M)^n$, and thus $\lambda\leq 2+M$. Hence, the non-trivial part is to prove the lower bound $\lambda \geq 2+M$. 

Fix $\eps>0$, $r\geq 1$ and let $i$ be an integer satisfying assumption \eqref{hypnice}. By concatenation of paths, we have 
\begin{equation}\label{minoration}
a^{(n)}_{i,i}\geq \left(u_r^{i\curvearrowright i}\right)^{\lfloor n/r\rfloor},
\end{equation}
where $u_r^{i\curvearrowright i}$ is the (weighted) number of paths of length $r$ going from $i$ to $i$, which stay above $i$. These paths necessarily stay confined in the strip $i\leq y\leq i+r$, and every vertex in this strip has a loop of weight greater than $M-\varepsilon$. 

Let us introduce $\overline{\env}$, the loop environment in which all loops have weight $M-\eps$. For arbitrary $i,r$, one has $u_r^{i\curvearrowright i} \geq d_r^{0\curvearrowright 0}$, where $d_r^{0\curvearrowright 0}$ represents the weighted number of paths that remain above $0$ in the graph $\mathbb{Z}^{(\overline{\env})}$.
The asymptotics of $d_r^{0\curvearrowright 0}$ can be derived using standard techniques of analytic combinatorics, through a slight generalization of the analysis of \emph{Motzkin numbers} (which correspond to the case of weights equal to $1$, see \cite[Example 6.3, p.~396]{Violet}).  We also refer to \cite{Spitzer}  for classical probabilistic treatments of random walks on the half-line, where most of the results we use are stated and proved. Since the computations will be useful later in the proof of \cref{lem:alpha_beta}, we provide the details below.  

For a loop environment ${\mathbf w}$, let ${H}_{i,i}^{[\geq i],\env}(z)$ be the generating function of paths in  $\mathbb{Z}^{(\env)}$ which start and end at $i$ and stay above $i$. In particular,
\begin{equation}\label{eq:positiveexcursion}
{H}_{i,i}^{[\geq i],\overline{\env}}(z)=\sum_{r\geq 0}d_r^{0\curvearrowright 0} z^r.
\end{equation}
By decomposing such paths with respect to successive passages at $i$ (for more details see \cite[Sec.V.4.1.]{Violet}, and especially the arch-decomposition), one can write
$$
{H}_{i,i}^{[\geq i],\overline{\env}}(z)=\frac{1}{1-(M-\varepsilon)z-z^2 {H}_{i,i}^{[\geq i],\overline{\env}}(z)}.
$$
By solving this equation, one finds
\begin{equation}\label{eq:FormuleD}
{H}_{i,i}^{[\geq i],\overline{\env}}(z)=\frac{1}{2z^2}\left(1-z(M-\varepsilon)-\sqrt{((M-\varepsilon)z-1)^2-4z^2}\right).
\end{equation}
In particular, ${H}_{i,i}^{[\geq i],\overline{\env}}(z)$ 
is $\Delta$-analytic (in the sense of \cite[Def.VI.1.]{Violet}) around its dominant singularity, which is located at $z^\ast=(2+M-\varepsilon)^{-1}$, and 
$$
{H}_{i,i}^{[\geq i],\overline{\env}}(z)\underset{z\to z^\ast}{=} (2+M-\varepsilon)- (2+M-\varepsilon)^2 (z^\ast-z)^{1/2}.
$$
Using the transfer theorem \cite[Cor.VI.1]{Violet}, we obtain that, for all $r\geq 0$,  
\begin{equation}\label{transfert}
d_r^{0\curvearrowright 0}\geq c(2+M-\varepsilon)^r r^{-3/2}, 
\end{equation}
for some $c>0$. We deduce from \eqref{minoration}, \eqref{transfert}, and $u_r^{i\curvearrowright i}\geq d_r^{0\curvearrowright 0}$ that
\begin{equation*}
\lambda =\limsup_{n\to\infty} (a_{i,i}^{(n)})^{1/n} \geq \liminf_{n\to+\infty} \left(\left(d_r^{0\curvearrowright 0}\right)^{\lfloor n/r\rfloor}\right)^{1/n}\geq \epsilon(r)(2+M-\varepsilon),
\end{equation*}
with $\epsilon(r)\to 1$ as $r\to\infty$. This yields $\lambda\geq 2+M-\varepsilon$, for any $\varepsilon>0$. This completes the proof.
\end{proof}

\subsection{The two extremal eigenvectors $\psi^+$ and $\psi^-$}

We describe in general the set of $\lambda$-eigenvectors for all $M$-nice environments.

\begin{prop}\label{solution}
Let $\env$ be a $M$-nice loop environment. Then the matrix $A$ defined in \cref{eq:A} has exactly two normalized positive extremal eigenvectors $\psi^{+}$ and $\psi^{- }$.

These are respectively non-decreasing and non-increasing and  
$$
\lim_{n\to\infty} \psi^{+}_n=\lim_{n\to-\infty} \psi^{-}_n=\infty.
$$
Furthermore, introduce  $L^-:=\inf\{n\in\mathbb Z : \ww_n<M\}$ and $L^+:=\sup\{n\in\mathbb Z : \ww_n<M\}$.
\begin{enumerate}
	\item If $L^-=-\infty$ then $\psi^+$  is increasing  and 
	$
	\displaystyle \lim_{n\to -\infty}\psi_n^+=0.
	$
	\item If $L^+=\infty$ then $\psi^-$  is decreasing  and 
	$\displaystyle 
	\lim_{n\to +\infty}\psi_n^-=0.
	$
		\item If $L^-\in \mathbb Z$ then $\psi^+$ is constant on $\{n\in \mathbb Z : n\leq L^-\}$  and increasing on $\{n\in \mathbb Z : n\geq L^-\}$.
	\item If $L^+\in \mathbb Z$ then $\psi^-$ is constant on $\{n\in \mathbb Z : n\geq L^+\}$  and decreasing on $\{n\in \mathbb Z : n\leq L^+\}$.
\end{enumerate}
 \end{prop}

As mentioned in the introduction, we will present later a fully explicit description of $\psi^+,\psi^-$. However, we still provide a proof of \cref{solution}, as the arguments are elementary and more robust than the combinatorial approach. In particular they could be generalized  to other kinds of one-dimensional graphs.

\begin{proof}[Proof of \cref{solution}]
To build $\psi^{+}$ and $\psi^{-}$, we proceed by truncation and approximation. Given $k\leq 0$ and $\varepsilon>0$, we define $\psi^{(k,\varepsilon)}$ as the unique sequence satisfying 
\begin{equation}\label{harmonic}
\forall n\geq k,\quad  \psi^{(k,\varepsilon)}_{n+1}+\ww_n \psi^{(k,\varepsilon)}_n+\psi^{(k,\varepsilon)}_{n-1} = \lambda\,\psi^{(k,\varepsilon)}_n,
\end{equation}
with the boundaries conditions
$$
\psi_n^{(k,\varepsilon)}=\left\{\begin{array}{ll}
0, & \text{for all $n\leq k-2$},\\
\eps, & \text{if $n=k-1$}, \\
2\eps, & \text{if $n=k$}.
\end{array}\right.	
$$

Then, if for some $n\geq k $, one has $0<\psi_{n-1}^{(k,\varepsilon)}<\psi_n^{(k,\varepsilon)}$, we get from  \eqref{harmonic}  that  
$$
\psi_{n+1}^{(k,\varepsilon)}\geq (\lambda-M)\psi_{n}^{(k,\varepsilon)}-\psi_{n-1}^{(k,\varepsilon)}>\psi_n^{(k,\varepsilon)}.
$$
Hence, $\psi^{(k,\varepsilon)}$ is non-decreasing on $\mathbb Z$ and increasing for $n\geq k-2$. Besides, it is clear that for $k\leq 0$, the function $\varepsilon\mapsto \psi^{(k,\varepsilon)}_0$ is continuous on $(0,\infty)$ and
$$
\lim_{\varepsilon\to 0} \psi^{(k,\varepsilon)}_0=0\quad\text{and}\quad 	\lim_{\varepsilon\to+ \infty} \psi^{(k,\varepsilon)}_0=+\infty. 
$$
We deduce that there exists $\varepsilon_k>0$ such that $\psi_0^{(k,\varepsilon_k)}=1$. 
In the sequel, we omit the dependence on $\varepsilon_k$ and simply write $\psi^{(k)}:=\psi^{(k,\varepsilon_k)}$.  

We now claim that for each $n\in\mathbb{Z}$, the set $A_n=\{\psi_n^{(k)} : k\leq 0 \}$ is bounded from below (by zero) but  also from above. This result is clear for $n\leq 0$ since $\psi^{(k)}$ is increasing and $\psi_0^{(k)}=1$, and it follows for all $n\geq 0$ by induction using \eqref{harmonic}. Consequently, by applying Cantor's diagonal argument, there exists a subsequence $(k_m)_{m\geq 0}$ and a sequence $\psi^+$ such that, for all $n\in\mathbb Z$,
$$
\lim_{m\to\infty}\psi_n^{(k_m)}=\psi_n^+.
$$

It is then straightforward that the limit point $\psi^+$ is non-decreasing, non-negative, and satisfies the linear recurrence equation \eqref{harmonic} for all $n\geq 0$, meaning $A\psi^+=\lambda \psi^+$. This eigenvector must be positive, otherwise it would be zero everywhere. 

Furthermore, for all $n\in\mathbb Z$,  
\begin{equation}\label{convexity}
\psi_{n+1}^+-\psi_{n}^+=(\lambda-1-\ww_n)\psi_n^+-\psi_{n-1}^+\geq \psi_{n}^+-\psi_{n-1}^+,
\end{equation}
with equality if and only if $\ww_n=M$. Since there exists $n_0$ such that $\ww_{n_0}<M$, we get that  $\psi_{n_0+1}^+-\psi_{n_0}^+>0$, and we deduce from the convex estimate \eqref{convexity} that $\lim_{n\to\infty}\psi_n^+=\infty$. 

Now, introduce $\ell=\lim_{n\to-\infty}\psi_n^+\geq 0$. We obtain that $(\lambda-2-\ww_n)\ell$ tends to $0$ as $n\to-\infty$. If $\{n\leq 0 : \ww_n<M\}$ is infinite, then necessarily $\ell=0$. Otherwise, there exists $L\in\mathbb Z$ such that for all $n\leq L$, $\ww_n=M$. It follows that $2\psi^+_n=\psi^+_{n-1}+\psi_{n+1}^+$ for all $n\leq L-1$. The only bounded solution of this equation is the constant solution, which implies the desired result. By symmetry, one can construct $\psi^{-}$ similarly. 

To conclude, observe that the vector space of real solutions of \eqref{harmonic} has dimension $2$, and since $\psi^{\pm}$ are clearly linearly independent, we deduce that $\psi^{\pm}$ are the extremal points of the convex set of positive solutions of \eqref{harmonic} normalized by $\psi_0=1$.
\end{proof}

\subsection{MERW on $\bbZ$: first properties} 

\label{sec:toy}

\begin{defi}[Mixtures of eigenvectors]\label{def:mixture}
Let $\env$ be an $M$-nice loop environment. For all $0\leq \kappa\leq 1$, let $\psi^{(\kappa)}$ be the positive eigenvector given by
$$
\psi^{(\kappa)}=\kappa \psi^+ + (1-\kappa)\psi^-.
$$
We denote by $X^{(\kappa)}$ the MERW associated to $\psi^{(\kappa)}$.
\end{defi}

The two processes $X^{(1)}$ and $X^{(0)}$, corresponding respectively to $\psi^+$ and $\psi^-$, will be referred to as the \emph{extremal} MERWs. They will also be denoted by $X^+$ and $X^-$ since $\psi^{(1)}=\psi^+$ and $\psi^{(0)}=\psi^-$. A direct consequence of \cref{solution} is the transience of the MERW for every nice environment and every positive $\lambda$-eigenvector. One can actually be more precise.

\begin{prop}\label{coro:transience_mixture}
	Let $\env$ be an $M$-nice loop environment. Then $A$ is $R$-transient, and for every $0\leq \kappa\leq 1$ and $k \in \mathbb{Z}$ the process $X^{(\kappa)}$ converges $\mathbb{P}^{\env}_k$-a.s. to $+\infty$ or $-\infty$. Furthermore, one has
	\begin{equation*}
	\mathbb{P}^{\env}_k\left(\lim_{n \to \infty} X_n^{(\kappa)} = +\infty\right) =
	1 - \mathbb{P}^{\env}_k\left(\lim_{n \to \infty} X_n^{(\kappa)} = -\infty\right) =
	\frac{\kappa \psi^+_k}{\kappa \psi^+_k + (1 - \kappa) \psi^-_k}.
	\end{equation*}
\end{prop}

\begin{proof}[Proof of \cref{coro:transience_mixture}]
From \cite[Sec.2]{vere1967ergodic}, we know that if there are several linearly independent solutions to $A\psi = \lambda \psi$, then $A$ is $R$-transient in the sense of \cref{rrec}.  In this case, it is not difficult to check that any MERW is transient, since one can write
\begin{equation*}
\sum_{n} \mathbb{P}_k(X_n^{(\kappa)} = k) \stackrel{\text{eq.\eqref{UniformPaths}}}{=} \sum_{n} \frac{a_{k,k}^{(n)}}{\lambda^n}\frac{\psi_k^{(\kappa)}}{\psi_k^{(\kappa)}}  = \sum_{n} \frac{a_{k,k}^{(n)}}{\lambda^n}  \stackrel{\text{eq.\eqref{rrec}}}{<} +\infty.
\end{equation*}

Therefore $X_n^{(\kappa)} \to \pm \infty$ almost-surely.
In order to compute $\mathbb{P}_k(X_n^{(\kappa)} \to + \infty)$  we use classical martingale arguments from potential theory for random walks.  

First consider the case $0<\kappa<1$. Note that the function $h(n)=\psi^+_n/\psi_n^{(\kappa)}$ is a positive and bounded harmonic function for the Markov chain $(X_n^{(\kappa)})_{n\geq 0}$. Indeed, for all $i\in\mathbb Z$, one has
$$\sum_{j \in \{i,i\pm 1\}} a_{i,j}\frac{\psi_j^{(\kappa)}}{\lambda \psi_i^{(\kappa)}}\frac{\psi^+_j}{\psi_j^{(\kappa)}}=\frac{\sum_{j \in \{i,i\pm 1\}} a_{i,j}\,{\psi_j^{+}}}{\lambda \psi_i^{(\kappa)}}=\frac{\psi_i^{+}}{\psi_i^{(\kappa)}}.$$

 Besides $\lim_{n\to-\infty} h(n)=0$ and $\lim_{n\to\infty} h(n)=1/\kappa$. 
In particular, $(h(X_n^{(\kappa)}))_{n\geq 0}$ is a bounded positive martingale and thus it converges almost surely. Applying the dominated convergence theorem we deduce that:
$$
\frac{\psi^+_k}{\kappa\psi^+_k+(1-\kappa)\psi_k^-}=
\mathbb{E}_k^\env\left[h(X_0^{(\kappa)})\right]=
\mathbb{E}_k^\env\left[h(X_\infty^{(\kappa)})\right]
=\frac{1}{\kappa}\mathbb{P}_k^{\env}\left(\lim_{t\to\infty} X_n^{(\kappa)}=+\infty\right).
$$

For the case $\kappa = 1$, consider, similarly to the previous situation, the function $\tilde{h}(n) = \psi^-_n / \psi^+_n$, so that $(\tilde{h}(X_n^+))_{n \geq 0}$ is a positive supermartingale, and hence $(\tilde{h}(X_n^+))_{n \geq 0}$ converges almost surely. Since $\lim_{n\to -\infty} \tilde{h}(n) = \infty$ and $\lim_{n\to \infty} \tilde{h}(n) = 0$, it follows that $(X_n^+)_{n \geq 0}$ necessarily converges to $+\infty$.  
The case $\kappa = 0$ is identical, using $\tilde{h}(n) = \psi^+_n / \psi^-_n$.

\end{proof}

\subsubsection*{A toy example: step-function loop environment}

In order to illustrate \cref{solution} and \cref{coro:transience_mixture}, we give the example of the \emph{step-function} environment $\env$ given by $\ww_i=M$ for all integers $i\leq 0$ and $\ww_i=0$ for all $i\geq 1$. Since $\env$ is $M$-nice, \cref{prop:nice} applies, and the combinatorial spectral radius satisfies $\lambda=M+2$. 

To obtain the expressions of $\psi^+$ and $\psi^-$, one can actually solve \eqref{harmonic1} directly. Let $\gamma$ be the unique root of $X^2-\lambda X+1$ in $(0,1)$. One can easily check that the corresponding two extremal eigenvectors are given by 
$$\psi_k^+=
\left\{\begin{array}{ll}
1, & \text{for $k\leq 0$},\\[5pt]
\frac{\gamma}{1+\gamma}\,\gamma^{-k} + \frac{1}{1+\gamma}\,\gamma^{k}, & \text{for $k\geq 0$},
\end{array}\right.
\quad\text{and}\quad 
\psi_k^-=
\left\{\begin{array}{ll}
1-(1-\gamma)k, & \text{for $k\leq 0$,}\\[5pt]
\gamma^{k}, & \text{for $k\geq 0$}.
\end{array}\right.$$

Following \cref{coro:transience_mixture}, the corresponding extremal MERWs $X^{+}$ and $X^{-}$ are respectively transient to $+\infty$ and $-\infty$ but exhibit distinct behaviors.  

First, since $\psi_n^+$ grows exponentially fast to $+\infty$, it is not difficult to see that $X^{(1)}$ has a positive linear speed. More precisely, one can check that the linear speed $v_M$ is given by  
$$\mathbb{E}[X_{n+1}^{+}-X_n^{+}|X_n^{+}=k]\underset{k\to+\infty}{\sim} v_M=\frac{1}{\lambda}\left(\frac{1}{\gamma}-\gamma\right).$$  
Besides, one can check that  $v_M {\sim} \sqrt{M}$ as $M\to 0$ while $v_M{\sim} M^{-1}$ as $M\to+\infty$.  

On the contrary, $X^{-}$ is a random walk with asymptotically zero drift on the negative integers. These kinds of random walks are also referred to as Bessel-like random walks or Lamperti processes. We refer to \cite{Duboux} for a more in-depth discussion.

\subsection{Combinatorial description of $\psi^+$ and $\psi^-$}

For general infinite matrices, there exists a combinatorial description of $\lambda$-eigenvectors in terms of paths (see \cite[Sec.2]{vere1967ergodic}). However, this expression may be difficult to compute explicitly in the general case. Regarding $\bbZ^{(\env)}$ with an $M$-nice environment $\env$, the simple structure of the weighted adjacency matrix $A$ allows us to provide some combinatorial expressions for the two extremal $\lambda$-eigenvectors.

Let $i,j,r$ be arbitrary integers. Similarly to \eqref{eq:positiveexcursion}, one denotes by 
${H}_{i,j}^{[\geq r],\env}$ and ${H}_{i,j}^{[\leq r],\env}$ the generating functions of lattice paths on $\mathbb{Z}^{(\env)}$ starting at $i$, ending at $j$, and staying respectively above or below the height $r$. Then, consider  
\begin{equation}\label{eq:def_alpha_beta}
\beta_i=\frac{1}{\lambda}H_{i,i}^{[\leq i],\env}\left(\frac{1}{\lambda}\right)
\quad\text{and}\quad 
\alpha_i=\frac{1}{\lambda}H_{i,i}^{[\geq i],\env}\left(\frac{1}{\lambda}\right).
\end{equation}

The convergences of series ${H}_{i,j}^{[\geq i],\env}(z)$ and ${H}_{i,j}^{[\leq i],\env}(z)$ at their dominant singularity $z=1/\lambda$ is a consequence of the fact that  $\mathbb{Z}^{(\env)}$ is $R$-transient (in the sense of \eqref{rrec}, see again, \cite{vere1967ergodic}). Finiteness of $\alpha_i,\beta_i$ also follow from elementary bounds that we give in the following lemma.

\begin{lem}\label{lem:alpha_beta}
 Let $\env$ be a  $M$-nice loop environment. For each $i\in\bbZ$,  $\alpha_i$ and $\beta_i$ are finite. More precisely,
\begin{equation}\label{item:encadrement_alpha} 
\gamma :=\frac{\lambda-\sqrt{\lambda^2-4}}{2} \leq  \alpha_{i},\beta_i \leq 1,
\end{equation}
and $\alpha_i=\gamma$ (resp. $\alpha_i=1$) if and only if $\ww_k=0$ (resp. $\ww_k=M$) for all $k\geq i$. Similarly, $\beta_i=\gamma$ (resp. $\beta_i=1$) if and only if $\ww_k=0$ (resp. $\ww_k=M$) for all $k\leq i$. 

Furthermore,  the  sequence $(\beta_i)_{i\in\mathbb Z}$ satifies the two sided  recurrence relation:
 	\begin{equation} \label{eq:beta}
 \forall i\in\mathbb Z ,\quad \beta_{i}=\frac{1}{\lambda-\ww_{i}-\beta_{i-1}} \quad \left(\text{{i.e.\@}}\quad \beta_{i-1}= \lambda-\ww_{i}-\frac{1}{\beta_{i}}  \right).
 \end{equation}  
Similarly, $(\alpha_i)_{i\in \mathbb Z}$ satisfies:
\begin{equation} \label{eq:alpha}
\forall i\in\mathbb Z ,\quad \alpha_{i}= \lambda-\ww_{i}-\frac{1}{\alpha_{i+1}}\quad \left(\text{{i.e.\@}}\quad \alpha_{i+1}=\frac{1}{\lambda-\ww_{i}-\alpha_i} \right).
\end{equation} 
\end{lem}	

\begin{proof}[Proof of  \cref{{lem:alpha_beta}}] 
Let $\overline{\mathbf w}$ (resp. $\underline{\mathbf w}$) be the environment where each loop has a constant weight $M$ (resp. $0$). By coefficient-wise domination, we have the following inequalities:
	$$
	{H}_{i,i}^{[\geq i],\underline{\mathbf w}}(1/\lambda) \leq {H}_{i,i}^{[\geq i],\mathbf w}(1/\lambda)\leq  {H}_{i,i}^{[\geq i],\overline{\mathbf w}}(1/\lambda).
	$$
	Similarly to \eqref{eq:FormuleD}, one can compute the lower and upper bounds  and we obtain
	$$
	\frac{\lambda^2-\lambda^2\sqrt{1-4/\lambda^2}}{2}
	\leq {H}_{i,i}^{[\geq i],\mathbf w}(1/\lambda)\leq 
	\frac{1-M/\lambda-\sqrt{(M/\lambda-1)^2-4/\lambda^2}}{2/\lambda^2}=
	\lambda.
	$$
Regarding the last equality in the latter equation, recall that $\lambda = M + 2$. Dividing by $\lambda$ then proves \eqref{item:encadrement_alpha}.

	We now turn to the proof of \eqref{eq:beta} and \eqref{eq:alpha}.
	Similarly to \cref{prop:nice}, the proof relies on path decompositions.
	Indeed, the arch-decomposition \cite[Sec.V.4.1.]{Violet} yields that
	\begin{equation*}\label{eq:Arch_i_i_plus_1}
	{H}_{i,i}^{[\geq i],\env}(z)=\frac{1}{1-z \ww_i-z^2 {H}_{i+1,i+1}^{[\geq i+1],\env}(z)}
	\quad\text{and}\quad  {H}_{i,i}^{[\leq i],\env}(z)=\frac{1}{1-z \ww_i-z^2{H}_{i-1,i-1}^{[\leq i-1],\env}(z)}.
	\end{equation*}
	Taking $z=1/\lambda$ in the above relations yields \eqref{eq:alpha} and \eqref{eq:beta}.
		
\end{proof}

\begin{rem}
	The recurrence relations satisfied by $\alpha_i$ and $\beta_i$ 
	(see equations~ \eqref{eq:beta} and~ \eqref{eq:alpha}) are related to what is known in the 
	physics literature as \emph{Riccati variables}. 
	Such variables arise naturally in the study of one-dimensional disordered systems 
	and Anderson localization; see, for instance, the classical work~\cite{Halperin1967} 
	and more recent developments~\cite{Comtet2010,Comtet_2013}. 
	Insights from this literature may prove useful 
	for understanding certain aspects of our model.
\end{rem}

\begin{rem}\label{rem:fraction_continue}
	For any $i\in\mathbb{Z}$, $\alpha_i$ only depends on loops $\{\ww_k, k\geq i\}$, while $\beta_i$ only depends on loops $\{\ww_k, k\leq i\}$. Besides, by iterating \eqref{eq:alpha} and \eqref{eq:beta}, one obtains the following continued fraction expansions:
	\begin{equation}\label{eq:alpha_fraction_continue}
	\alpha_i=\cfrac{1}{\lambda-\ww_i-\cfrac{1}{\lambda-\ww_{i+1}-\cfrac{1}{\ddots}}}\quad\text{and}\quad 
	\beta_i=\cfrac{1}{\lambda-\ww_i-\cfrac{1}{\lambda-\ww_{i-1}-\cfrac{1}{\ddots}}}.
	\end{equation}
	
	Note that $\alpha_i$ and $\beta_i$ are obtained by randomly iterating the functions $g_s$, for $0 \leq s \leq M$, defined in \eqref{iterated} and represented on $[\gamma,1]$ in Figure \ref{fig:g_0_g_M}. The values $\gamma$ and $1$ in \eqref{item:encadrement_alpha} are the respective fixed points of $g_0$ and $g_M$.
\end{rem}	

\begin{figure}[H]
	\centering
	\begin{minipage}{0.7\textwidth}
		\centering
		\includegraphics[scale=0.35]{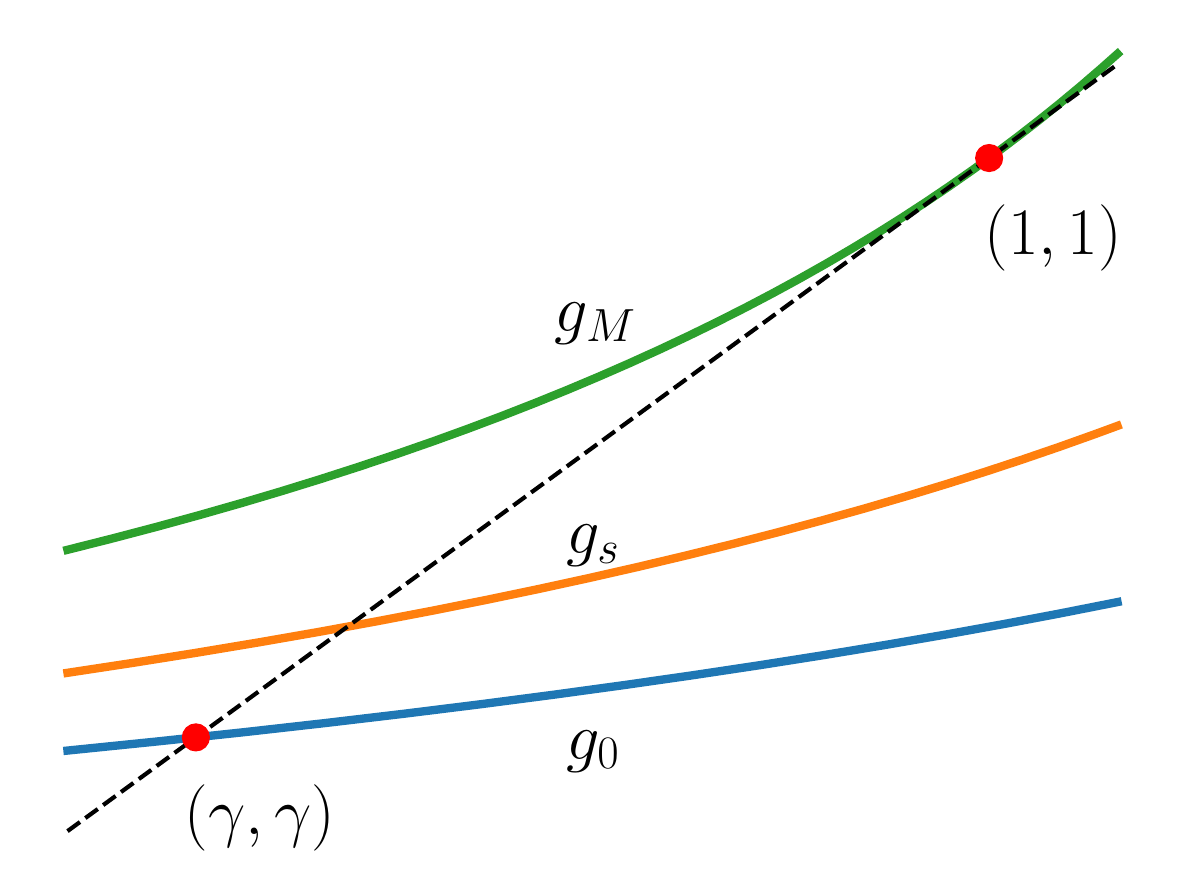}
		\caption{Sketchs of plots of $g_s$.}
		\label{fig:g_0_g_M}
	\end{minipage}\hfill
	\begin{minipage}{0.3\textwidth}
		\begin{equation}\label{iterated}
		\hspace{-4cm}g_s(x)=\frac{1}{\lambda-s-x}.
		\end{equation}
	\end{minipage}
\end{figure}

We now can state the main result of \cref{sec:Deterministic}.

\begin{thm}\label{Th:psi_combi}
Let $\env$ be an $M$-nice loop environment, and let $(\alpha_i)_{i\in\bbZ}$ and $(\beta_i)_{i\in\bbZ}$ be defined by \eqref{eq:def_alpha_beta}.  The two normalized extremal $\lambda$-eigenvectors of $A$, $\psi^+$ and $\psi^-$ defined in \cref{solution}, have the following representations:

	\begin{equation}\label{eq:psi_combi}
	\psi_i^+=\begin{cases}
	\beta_{-1}\dots \beta_{i},&\text{for }i<0,\\[5pt]
	1, &\text{for } i=0,\\[5pt]
	{(\beta_0\beta_1\beta_{2}\dots \beta_{i-1})^{-1}},&\text{for }i>0,
	\end{cases}
	\qquad	
	\psi_i^{-}=\begin{cases}
	{(\alpha_0\alpha_{-1}\dots \alpha_{i+1})^{-1}},&\text{for }i<0,\\[5pt]
	1, &\text{for }i=0,\\[5pt]
	\alpha_1\alpha_{2}\dots \alpha_{i},&\text{for }i>0.
	\end{cases}
	\end{equation}
 \end{thm}	

\begin{proof}[Proof of  \cref{Th:psi_combi}] 
Let us check that the expression of  $\psi^{+}$ given by \eqref{eq:psi_combi} indeed defines a $\lambda$-eigenvector.
For all $i\in\mathbb Z$, one has  
	\begin{equation*}
	\psi^+_{i+1}+\ww_i\psi^+_i+\psi^+_{i-1}
	\stackrel{\text{\eqref{eq:psi_combi}}}{=}
	\frac{1}{\beta_{i}}\psi^+_{i}+\ww_i\psi^+_i+\beta_{i-1}\psi^+_{i}
	\stackrel{\text{\eqref{eq:beta}}}{=}\left(\frac{1}{\beta_{i}} +\ww_i + \lambda-\ww_{i}-\frac{1}{\beta_{i}}\right)\psi^+_{i}=\lambda\psi^+_i.
	\end{equation*}
Similarly, one can check that $\psi^-$ is a $\lambda$-eigenvector. 

It remains to check that $\psi^+$ and $\psi^-$, defined by \eqref{eq:psi_combi}, are not collinear. Assume, for contradiction, that $\psi^+$ and $\psi^-$ are collinear. Then, they must be equal since $\psi^+_0=\psi^-_0=1$, and this implies by an immediate induction that $\alpha_i=1/\beta_{-i}$ for every $i\in\mathbb{Z}$. By \eqref{item:encadrement_alpha}, this means that $1=\alpha_i=\beta_{-i}$ for every $i\in\mathbb{Z}$. Finally, \cref{lem:alpha_beta} yields that $\ww_i=M$ for every $i$, which contradicts the assumption that $\env$ is $M$-nice.
\end{proof}

%

\section{MERW on $\mathbb{Z}$: \emph{i.i.d.\@} loop environments}
\label{sec:Random}

Throughout this section, we assume that the loop environment $\env=(\ww_i)_{i\in \mathbb{Z}}$ is given by a sequence of \emph{i.i.d.\@} random variables with common distribution ${\nu}$. 

We shall denote by $\mu$ the distribution of the environment, \emph{i.e.\@}, the product measure $\mu=\nu^{\otimes \mathbb{Z}}$, and by $\mathbb{E}_\mu$ the corresponding expectation.

\begin{hypothesis}[Hyp-$\nu$]\label{hyp:nu} There exists a constant $M>0$ such that $\nu([0,M])=1$, $\nu\neq \delta_M$, and $M\in {\rm supp}\ \nu$, that is, for all $\varepsilon>0$, $\nu([M-\varepsilon,M])>0$.
\end{hypothesis}

Under Hypothesis \hyperref[hyp:nu]{Hyp-$\nu$}, the random loop environment $\env$ is $\mu$-almost surely $M$-nice. In particular, \cref{prop:nice} implies that the random combinatorial spectral radius $\lambda$ is $\mu$-almost surely constant and equal to $2+M$.
	
We recall that $\mathbb{P}^{\env}_k$ denotes the \emph{quenched} probability distribution over the space $\mathbb{Z}^{\bbZ_{\geq 0}}$ of trajectories, starting from $k\in\mathbb{Z}$, endowed with the cylinder $\sigma$-algebra. This corresponds to the probability distribution of trajectories when the environment $\env$ is fixed. In contrast, we introduce the \emph{annealed} probability on the space $[0,M]^\bbZ\times \mathbb{Z}^{\bbZ_{\geq 0}}$ of pairs (environment, trajectory), defined by  
$$
\mathbb{Q}_k(W\times A)=\int_W \mathbb{P}^\env_k(A) \mathrm{d}\mu(\env)=\mathbb{E}_\mu[\mathds{1}_W \mathbb{P}_k^{}(A)],
$$
for any $W$ and $A$ in the corresponding cylinder $\sigma$-algebra.
We shall denote by $\mathbb{E}^{\env}_k$ (resp. $\mathbb{E}^{\mathbb{Q}}_k$) the expectation corresponding to $\mathbb{P}^{\env}_k$ (resp. to $\mathbb{Q}_k$).

\subsection{MERW associated with the random extremal eigenvectors}
\label{soussection:MERWextremale}

It follows from \cref{lem:alpha_beta} and \cref{rem:fraction_continue} that  $(\beta_i)_{i\in\mathbb{Z}}$ and $(\alpha_{i})_{i\in\mathbb{Z}}$ are stationary sequences of random variables in $(\gamma,1)$. Besides, together with \eqref{eq:beta}, the sequence $\left(\beta_{i}\right)_{i\in\bbZ}$ is a stationary and ergodic Markov chain. Observe that $\left(\alpha_{i}\right)_{i\in\bbZ}$ is also stationary, but it is \emph{not} a Markov chain, contrary to the time-reversed sequence $\left(\alpha_{-i}\right)_{i\in\bbZ}$. 

To analyze the MERW corresponding to the (random) extremal eigenvectors $\psi^+$ and $\psi^-$, we need a few formal notations. For any $\ell\in\mathbb{Z}$, introduce the shift $\theta^\ell$, which operates on environments:
$$
\begin{array}{r r c l}
\theta^\ell : & [0,M]^\bbZ & \to & [0,M]^\bbZ\\
& \left(\ww_k \right)_{k\in\mathbb{Z}} & \mapsto & \left(\ww_{k+\ell} \right)_{k\in\mathbb{Z}}.
\end{array}
$$

The explicit expression of $\psi^+$ implies (see \eqref{eq:ProbasTransition} and \cref{Th:psi_combi}) that the corresponding random transition probabilities of the quenched MERW $X^+$ take the following form:
\begin{equation}\label{eq:transition_p_plus}
p_{i,i+1}^{+}(\env)=\frac{1}{\lambda\beta_i},\quad
p_{i,i-1}^+(\env)= \frac{\beta_{i-1}}{\lambda }\quad \text{and}\quad p_{i,i}^+(\env)= \frac{\ww_i}{\lambda},
\end{equation}
for arbitrary $i\in\mathbb{Z}$. Obviously, a similar expression can be obtained for the second extremal quenched MERW $X^-$. When there is no ambiguity, we sometimes shorten $p_{i,j}^+(\env)$ to $p_{i,j}^+$. One can easily  check that  
\begin{equation}\label{eq:stationarykernel}
p_{i,j}\left(\theta^\ell(\env)\right)=p_{i+\ell,j+\ell}(\env),
\end{equation}
for all $i,j,\ell\in\mathbb Z$. It follows that the process $X^+$ is a random walk in a $\mu$-ergodic (with respect to the shift $\theta$) random environment. Many quantitative results are known for such processes, we allude for instance to \cite{Zeitouni} and \cite{Alili_RWRE}. 

A particularly relevant quantity (given in the two latter references) is the random variable (defined on the environment space) given by  
\begin{equation}\label{eq:S_barre}
S=\frac{1}{p_{0,1}^+}+\sum_{i=1}^{+\infty} \frac{1}{p_{-i,-i+1}^+}\prod_{j=0}^{i-1}\rho_{-j},
\quad\text{where}\quad  \rho_i=\frac{p_{i,i-1}^+}{p_{i,i+1}^+}=\beta_i\beta_{i-1}.
\end{equation}

It is known (see, for instance, \cite[Sec. 2.1]{Zeitouni}) that if $\mathbb{E}_\mu[\log(\rho_0)]<0$, then the corresponding random walk is $\mathbb{Q}_k$-a.s.\@ transient to $+\infty$ for all starting points $k$. Here, since $\beta_i<1$ $\mu$-a.s., we get from the right-hand side of \eqref{eq:S_barre} that indeed $\mathbb{E}_\mu[\log(\rho_0)]<0$. Hence, we recover the fact that $X^+$ is transient to $+\infty$ for $\mu$-almost all environments (recall \cref{coro:transience_mixture}).

One can actually be much more precise. We have that (see \cite[Th.2.1.9]{Zeitouni}) $X^+$ has a constant positive linear speed $v$ with $\mathbb{Q}_k$-probability one for any $k\in\mathbb{Z}$ if and only if $\mathbb{E}_\mu[S]<\infty$. In that case, one has  
\begin{equation}\label{eq:drift}
v=\frac{1}{\mathbb{E}_\mu[S]} >0.
\end{equation}

\begin{thm}\label{th:transient}
Assume Hypothesis \hyperref[hyp:nu]{Hyp-$\nu$}. For all $k\in\mathbb{Z}$, one has
$$
\lim_{n\to\infty}\frac{X_n^+}{n} = v\quad\text{and}\quad
\lim_{n\to\infty}\frac{X_n^-}{n} = -v.
$$
The above limits hold $\mathbb{Q}_k$-a.s.\@ and thus $\mathbb{P}_k^\env$-a.s.\@ for $\mu$-almost all loop environments $\env$.
\end{thm}

\begin{proof}[Proof of \cref{th:transient}]
Using \eqref{eq:transition_p_plus}, one can rewrite $S$ as follows:
\begin{equation}
S= \sum_{i=1}^{+\infty} \lambda\beta_{-i} \prod_{j=0}^{i-1}\beta_{-j}\beta_{-j-1} +\lambda\beta_0\notag\\
=\lambda \beta_0+\lambda \sum_{i=1}^{+\infty} \beta_0\beta_{-1}^2\beta_{-2}^2\dots \beta_{-i}^2\label{eq:Stilde}.
\end{equation}

We use \cref{lem:alpha_beta} in order to bound the $\beta_i$'s. Let us write $\beta_{i}=g_{\ww_i}(\beta_{i-1})$, where $g_s(x)$ is introduced in \eqref{iterated}.
Let $\delta>0$ be such that $\mu(\{\ww_{i}\leq M-\delta\})>0$ (Hypothesis \hyperref[hyp:nu]{Hyp-$\nu$}). Using the fact that $g_s(x)$ is increasing in both $x$ and $s$ (see \cref{fig:g_0_g_M}) and that $g_M(1)=1$, we obtain
\begin{align}
\beta_i&\leq \mathds{1}_{\{\ww_i\leq M-\delta\}}\, g_{M-\delta}(\beta_{i-1})+\mathds{1}_{\{\ww_i > M-\delta\}}\, g_M(1)\notag\\
&\leq 1-\mathds{1}_{\{\ww_i\leq M-\delta\}}\, (1-g_{M-\delta}(1))=: Z_i,\label{eq:def_Zi}
\end{align}
where the $Z_i$'s are \emph{i.i.d.\@} and such that $\mathbb{E}_{\mu}[Z_i]<1$ and $\mathbb{E}_{\mu}[Z_i^2]<1$. Then, we get, for all $i\geq 0$,
$$
\mathbb{E}_{\mu}\left[\beta_0\beta_{-1}^2\beta_{-2}^2\dots \beta_{-i}^2\right]
\leq  \mathbb{E}_{\mu}\left[Z_0Z_{-1}^2Z_{-2}^2 \dots Z_{-i+1}^2Z_{-i}^2 \right]\\
\leq  \mathbb{E}_{\mu}\left[Z_{0}^2\right]^{i}.
$$
Hence, we obtain $\mathbb{E}_{\mu}[{S}]<+\infty$. Using \cite[Th.2.1.9]{Zeitouni}, this implies that $X^+$ is $\mathbb{P}^\env$-a.s.\@ and $\mathbb{Q}_k$-a.s.\@ transient with linear speed $v$. 
The result for $X^-$ can be deduced by symmetry.\end{proof}

We now provide a more explicit form for the random variable $S$ involved in $v$. It will be particularly useful when analyzing the case of Bernoulli environments in \cref{sec:Bernoulli}. 

Recall that $a_{i,j}^{(n)}$ denotes the weighted number of paths in $\mathbb{Z}^{(\env)}$ and $X^{(\kappa)}$  is the MERW associated with  $\psi^{(\kappa)}=\kappa\psi^{+}+(1-\kappa)\psi^-$  In the following, we sometimes  explicitly write the dependence on $\env$ to make it more transparent, as for the $\alpha_i$ and the $\beta_i$'s.

\begin{prop}[Combinatorial expression for $S$]\label{lemScombinatoir}
	Under  Hypothesis \hyperref[hyp:nu]{Hyp-$\nu$}, one has 
\begin{equation}\label{eq:invdriftexpressions}
	S(\env)=\sum_{n=0}^\infty\mathbb P_0^\env(X_n^+=0),
	\end{equation}
and, for every $i,k\in\mathbb Z$ such that $i\geq k$, 
	\begin{equation}\label{eq:formuleSbis}
	\sum_{n\geq 0}\mathbb{P}^\env_{k}(X^{(\kappa)}_n=i)=\frac{\psi^{(\kappa)}_{i}(\env)}{\psi^{(\kappa)}_{k}(\env)}\sum_{n\geq 0}\frac{a^{(n)}_{k,i}(\env)}{\lambda^n}=
	\frac{\psi_i^{(\kappa)}(\env)}{\psi_k^{(\kappa)}(\env)}\frac{\lambda  \psi^-_{i-k}(\theta^k(\env))}{\lambda-\ww_k-\beta_{-1}(\theta^k(\env))-\alpha_1(\theta^k(\env))}.
	\end{equation}
\end{prop}

\begin{rem}\label{rem:joliS}
	Combining \eqref{eq:invdriftexpressions}  and \eqref{eq:formuleSbis} (for $i=k=0$) in \cref{lemScombinatoir} yields
	\begin{equation}\label{eq:joliS}
	S(\env)=\sum_{n=0}^\infty \frac{a_{0,0}^{(n)}(\env)}{\lambda^n}=\frac{\lambda}{\lambda-\ww_0-\beta_{-1}(\env)-\alpha_1(\env)}.
\end{equation}
\end{rem}

\begin{proof}[Proof of  \cref{lemScombinatoir}] We begin with the proof of  \eqref{eq:invdriftexpressions}.	From \cite[Lem.2.1.12]{Zeitouni}) we have that,  for every $k\in\bbZ$, $\mathbb E^{\mathbb Q}_k[T_{k+1}]=\mathbb E_\mu[S]$, where 
	$$
	T_{k+1}=\inf\{n\geq 0 : X_n^+=k+1\}.
	$$ 
Besides, following the same lines as \cite[p. 338]{Alili_RWRE}, one can write
	$$\mathbb E^{\mathbb Q}_0[T_1]=\int \sum_{i\leq 0} \mathbb E^{\theta^{-i}(\env)}_0[N_i]\;\mu(d\env),\quad\text{where}\quad  N_i={\rm card}(\{0\leq n<T_1 : X_n^+=i\}).$$
	To go further, set for all $k\geq 0$, $$N_0^{k+1}={\rm card}(\{0\leq n<T_{k+1} : X_n^+=0\})\quad \text{and}\quad N_0^{\infty}={\rm card}(\{n\geq 0 : X_n^+=0\}).$$ 
	By using the strong Markov property, we obtain that 
	$$\sum_{i\leq 0}\mathbb E_0^{\theta^{-i}(\env)}[N_i]=\sum_{i\leq 0}\mathbb E_{i}^\env[N_0^{i+1}]=\mathbb E_0^\env[N_0^\infty].$$
	Therefore, \eqref{eq:invdriftexpressions} is then a simple consequence of the latter equality.
	
It remains to prove \eqref{eq:formuleSbis}. By using \eqref{UniformPaths}, let us write
	\begin{equation}
	\sum_{n\geq 0}\mathbb{P}^\env(X^{(\kappa)}_n=i)
	=\sum_{n\geq 0}a^{(n)}_{k,i}\frac{\psi^{(\kappa)}_{i}}{\lambda^{n}\psi^{(\kappa)}_{k}}
	=\frac{\psi^{(\kappa)}_{i}}{\psi^{(\kappa)}_{k}}H^\env_{k,i}\left(\frac{1}{\lambda}\right),\label{eq:P_Xt}
	\end{equation}
	where $H^\env_{k,i}$ is the generating series of paths going from $k$ to $i$ in the lattice $\bbZ^{(\env)}$. Such a path can be uniquely decomposed as a sequence of excursions above or below $k$ and, if $i>k$, a final positive path going from $k$ to $i$ (never returning to $k$). 
	
	Again, following \cite[Sec.V.4.1.]{Violet}, we obtain that 
	\begin{equation}\label{eq:H0i}
	H^\env_{k,i}(z)=
	\begin{cases}
	\displaystyle\frac{z   H_{k+1,i}^{[\geq k+1],\env}(z)}{1-z\ww_k-z^2H_{k-1,k-1}^{[\leq k-1],\env}(z)-z^2H_{k+1,k+1}^{[\geq k+1],\env}(z)}  , &\text{if } i>k,\\[20pt]
	\displaystyle \frac{1}{1-z\ww_k-z^2H_{k-1,k-1}^{[\leq k-1],\env}(z)-z^2H_{k+1,k+1}^{[\geq k+1],\env}(z)}, &\text{if } i=k.
	\end{cases}
	\end{equation}
	Besides, from the \emph{last-passage decomposition} \cite[Eq.(51),  p.320]{Violet} one has
	$$
	{H}_{k,i}^{[\geq k],\env}\left(z\right)={H}_{k+1,k+1}^{[\geq k+1],\env}\left(z\right) z {H}_{k+2,k+2}^{[\geq k+2],\env}\left(z\right) z \dots
 {H}_{i-1,i-1}^{[\geq i-1],\env}\left(z\right) z {H}_{i,i}^{[\geq i],\env}\left(z\right).
	$$
	Plugging $z=1/\lambda$ into the last equality yields 
	$$
	\frac{1}{\lambda}{H}_{k,i}^{[\geq 1],\env}\left(\frac{1}{\lambda} \right)=\alpha_{k+1}(\env)\dots \alpha_i(\env)=\frac{\psi_i^{-}(\env)}{\psi_k^-(\env)}.
	$$ 
	Combining \eqref{eq:P_Xt} and  \eqref{eq:H0i} finally gives
	\begin{align*}
	\sum_{n\geq 0}\mathbb{P}^\env_{k}(X^{(\kappa)}_n=i)&=
	\frac{\psi_i^{(\kappa)}(\env)}{\psi_k^{(\kappa)}(\env)} \frac{\lambda \psi^-_{i}(\env)/ \psi^-_{k}(\env)}{\lambda-\ww_k-\beta_{k-1}(\env)-\alpha_{k+1}(\env)}\\
	&=\frac{\psi_i^{(\kappa)}(\env)}{\psi_k^{(\kappa)}(\env)} \frac{\lambda \psi^-_{i-k}(\theta^k(\env))}{\lambda-\ww_k-\beta_{-1}(\theta^k(\env))-\alpha_1(\theta^k(\env))}.	
	\end{align*}

In the last equality, we use 
\begin{equation}\label{eq:alphabetashift}
\beta_{j}(\theta^k(\env))=\beta_{j+k}(\env), \;
\alpha_{j}(\theta^k(\env))=\alpha_{j+k}(\env)
\quad\text{and}\quad\frac{\psi_{i}^{\pm}(\env)}{\psi_{k}^{\pm}(\env)}=\psi_{i-k}^{\pm}(\theta^k(\env)),
\end{equation}
which hold for every $j,k\in\mathbb{Z}$. These equalities are simple consequences of \eqref{eq:alpha_fraction_continue} and can be easily understood as a consequence of the translation invariance of the model. 

This completes the proof.

\end{proof}

\subsection{Analysis of MERW associated with a non-extremal eigenvector}

In the general case, the MERW associated with a generic random eigenvector $\psi$ (which, according to \cref{solution}, must be a convex combination of $\psi^+$ and $\psi^-$) is not a random walk in a usual ergodic environment. Indeed, let  $\psi^{(\kappa)}(\env)=\kappa\psi^+(\env)+(1-\kappa)\psi^-(\env)$ as in Definition \ref{def:mixture}. 
One can check, by using \eqref{eq:alphabetashift}, that 
\begin{equation}
\frac{\psi_{j}^{(\kappa)}(\env)}{\psi_{k+i}^{(\kappa)}(\env)}\neq \frac{\psi_{j}^{(\kappa)}(\theta^k(\env))}{\psi_{i}^{(\kappa)}(\theta^k(\env))},
\end{equation}
when $0<\kappa<1$. In particular, the corresponding Markov kernel is not stationary with respect to the shift, as in \eqref{eq:stationarykernel}.

However, one can still describe the behavior of the corresponding MERW by comparison with the extremal cases. This is the purpose of the following result.

\begin{thm}\label{th:mixture}
Let $0<\kappa< 1$ and $X^{(\kappa)}$ be the MERW introduced in \cref{def:mixture} in an i.i.d.\@ random loop environment satisfying \hyperref[hyp:nu]{Hyp-$\nu$}. For $\mu$-almost all environments $\env$ and all $k\in\mathbb{Z}$,   
$$
\mathbb{P}_k^\env
\left(\lim_{n\to\infty} X_n^{(\kappa)}=+\infty\right)
+\mathbb{P}_k^\env
\left(\lim_{n\to\infty} X_n^{(\kappa)}=-\infty\right)
=1.
$$
More precisely, on the event that $\lim_{n\to \infty} X_n^{(\kappa)}= +\infty$ (resp. $\lim_{n\to \infty} X_n^{(\kappa)}= -\infty$), one has
$$
\lim_{n\to+\infty}\frac{X_n^{(\kappa)}}{n} = v,\qquad
\left(\text{resp. }\lim_{n\to+\infty}\frac{X_n^{(\kappa)}}{n} = -v\right).
$$
\end{thm}

\begin{proof}[Proof of \cref{th:mixture}]
First, note that $X_n^{(\kappa)}\to \pm \infty$ is a consequence of \cref{coro:transience_mixture}, since $\env$ is $\mu$-almost surely $M$-nice. We now prove that $X^{(\kappa)}$ has linear speed when $X_n^{(\kappa)}\to +\infty$. The proof when $X_n^{(\kappa)}\to -\infty$ can be deduced by symmetry.

Denote by $p_{i,j}^{(\kappa)}$ the transition probabilities corresponding to the MERW $X^{(\kappa)}$, that is
\begin{equation}\label{eq:p_kappa}
p_{i,i+1}^{(\kappa)}= \frac{1}{\lambda}\frac{\kappa \psi^+_{i+1}+(1-\kappa)  \psi^-_{i+1}}{\kappa \psi^+_{i}+(1-\kappa)  \psi^-_{i}},
\quad
p_{i,i}^{(\kappa)}= \frac{\ww_i}{\lambda},\quad
p_{i,i-1}^{(\kappa)}=
\frac{1}{\lambda}\frac{\kappa \psi^+_{i-1}+(1-\kappa)  \psi^-_{i-1}}{\kappa \psi^+_{i}+(1-\kappa)  \psi^-_{i}}.
\end{equation}
The idea is to exploit the fact that, for large $i$, since $\psi^+_i$ is much greater than $\psi^-_i$, $p_{i,j}^{(\kappa)}$ is very close to $p_{i,j}^+$.
 A quick computation\footnote{We use the inequality $|(a+\eps)/(b+\delta)-a/b|\leq 2\max\{a,b\}\max\{\eps,\delta\}/b^2$, which holds for every $a,b,\eps,\delta>0$.} indeed shows that, for every $i\geq 0$ and every $j\in\{i-1,i,i+1\}$,
\begin{equation}\label{eq:BornePourCouplage}
\left| p_{i,j}^{(\kappa)}-p_{i,j}^+ \right| \leq \frac{2(1-\kappa)}{\kappa}\frac{\psi^-_{i-1}}{\psi^+_i}.
\end{equation}

Let $(U_n)_{n\geq 0}$ be a sequence of \emph{i.i.d.\@} uniform random variables in $(0,1)$, independent of $\env$. It is not difficult to see that $X^{(\kappa)}$ can be constructed inductively as follows. We set $X_0^{(\kappa)}=k$ and, for all $n\geq 0$ and $i\in\mathbb{Z}$ such that $X_n^{(\kappa)}=i$,  
$$
X_{n+1}^{(\kappa)}=
\begin{cases}
i+1,&\text{if } U_n<p_{i,i+1}^{(\kappa)},\\[5pt]
i,&\text{if } U_n<p_{i,i+1}^{(\kappa)}+p_{i,i}^{(\kappa)},\\[5pt]
i-1,&\text{otherwise.}
\end{cases}
$$
Note that $X^+$ can also be obtained by  replacing the transition probabilities $p^{(\kappa)}_{i,j}$ with $p^+_{i,j}$. 

We say that $n$ is a \emph{bad} time for $X^{(\kappa)}$ if
$$
U_n\in (p_{i,i+1}^{\kappa},p_{i,i+1}^{+}) \cup (p_{i,i-1}^{+},p_{i,i-1}^{\kappa}),\quad \text{with}\quad  i=X_n^{(\kappa)}.
$$
If $n$ is not a bad time, then $X^{(\kappa)}$ evolves between time $n$ and $n+1$ exactly as $X^+$ would. We denote by $\mathcal{B}$ the random set of bad times. For every $K\in\mathbb{Z}$, introduce the event
$$
\mathcal{E}_K=\left\{ \inf_{n\geq 0} X_n^{(\kappa)}\geq K \right\}.
$$
Note that $\bigcup_{K\leq k}\mathcal{E}_K=\{ X_n^{(\kappa)}\to+\infty\}$.

In the sequel, to lighten notation, we shall omit $\kappa$ and write $X_n$ instead of $X_n^{(\kappa)}$. Thereafter, for any $K\leq k$, one can write 
\begin{multline}
\mathbb{E}^\env_k[\mathrm{card}(\mathcal{B})\mathbf{1}_{\mathcal{E}_K}]=  \sum_{i< k}\sum_{n\geq 0}\mathbb{E}^\env_k[\mathbf{1}_{\{X_n=i\}}\mathbf{1}_{\{n\in \mathcal{B}\}}\mathbf{1}_{\mathcal{E}_K}]
+  \sum_{i\geq k}\sum_{n\geq 0}\mathbb{E}^\env_k[\mathbf{1}_{\{X_n=i\}}\mathbf{1}_{\{n\in\mathcal{B}\}}\mathbf{1}_{\mathcal{E}_K}]\\
\leq 
\mathbb{E}^\env_k\Bigg[\sum_{n\geq 0}\mathbf{1}_{\{K\leq X_n<k\}}\Bigg] + \sum_{i\geq k}\sum_{n\geq 0}\mathbb{E}^\env_k[\mathbf{1}_{\{n\in \mathcal{B}\}}\mathbf{1}_{\{X_n=i\}}].\label{maj1}
\end{multline}

The first term on the right-hand side of \eqref{maj1} is the expected number of visits to the finite set $\{K,\dots,k-1\}$ and is therefore finite, since for each fixed environment, $X_n$ is a transient Markov chain. 

Furthermore, since for all $n\geq 0$, one has $U_n$ independent of $X_n$, and ${\psi_i^{(\kappa)}(\env)}\leq \kappa {\psi^+_i(\env)}$, we get from \eqref{eq:BornePourCouplage}, \cref{lemScombinatoir}, and the Cauchy-Schwarz inequality that
\begin{align}
\sum_{i\geq k}\sum_{n\geq 0}\mathbb{E}^\env_k[\mathbf{1}_{\{n\in \mathcal{B}\}}\mathbf{1}_{\{X_n=i\}}]
&\leq\sum_{i\geq k}   \frac{2(1-\kappa)}{\kappa}\frac{\psi^-_{i-1}(\env)}{\psi^+_i(\env)} \sum_{n\geq 0}\mathbb{P}^\env_k(X_n=i)\notag\\
&\leq  C_k(\env) \left(\sum_{i\geq k} \psi^-_{i-1}(\env)\psi^-_{i-k}(\theta^k\env)\right)\notag\\
&\leq  C_k(\env) \sqrt{\sum_{i\geq k} (\psi^-_{i-1}(\env))^2}\sqrt{\sum_{i\geq k}(\psi^-_{i-k}(\theta^k(\env)))^2},\label{eq:CauchySchwartz}
\end{align}
for some positive constant $C_k(\env)$ depending on $\kappa, k$, and $\env$. Besides, the very same domination arguments as in \eqref{eq:def_Zi} allow us to check that 
$$
\mathbb{E}_\mu\left[\sum_{i\geq k}\left(\psi_{i-1}^-\right)^2\right]<\infty\quad\mbox{and}\quad \mathbb{E}_\mu\left[\sum_{i\geq k}\left(\psi_{i-k}^-(\theta^k(\cdot))\right)^2\right]=\mathbb{E}_\mu\left[\sum_{i\geq 0}\left(\psi_{i}^-\right)^2\right]<\infty.
$$

We deduce that the right-hand side of \eqref{eq:CauchySchwartz} is finite for $\mu$-almost all environments $\env$. This proves that
$\mathbb{E}^\env_k[\mathrm{card}(\mathcal{B})\mathbf{1}_{\mathcal{E}_K}]<\infty$ for $\mu$-almost all environments $\env$, and we obtain that $\mathrm{card}(\mathcal{B})<\infty$ $\mathbb{P}_k^\env$-a.s.\@ on the event $\{X_n\to+\infty\}$.

Introduce $\tau=\sup \mathcal{B}$ and consider, for any $m\in\mathbb{Z}$, the event $\Lambda_{m}=\{\tau<m\}\cap \{X_n\to+\infty\}$. Then, by setting $X^+_0=X_m^{(\kappa)}$ and, for all $n\geq 0$ and $i\in\mathbb{Z}$ such that $X_n^{+}=i$, 
$$
X_{n+1}^{+}=
\begin{cases}
i+1,&\text{if } U_{n+m}<p_{i,i+1}^{(\kappa)},\\[5pt]
i,&\text{if } U_{n+m}<p_{i,i+1}^{\kappa}+p_{i,i}^{(\kappa)},\\[5pt]
i-1,&\text{otherwise,}
\end{cases}
$$
one has a coupling between $\left(X_n\right)_{n\geq 0}$ and $\left(X_n^+\right)_{n\geq 0}$ such that $X_{n+m}=X^+_{n}$ for all $n\geq 0$ on $\Lambda_m$. 

As a consequence, $X$ has a linear speed $v^+$ on $\Lambda_m$ for every $m\geq 0$. Since $\tau<\infty$, we deduce the same result on the event $\{X_n\to +\infty\}$. The proof for $\{X_n\to -\infty\}$ can be deduced by symmetry. \end{proof}

\subsection{The case of a Bernoulli \emph{i.i.d.\@} loop environment}
\label{sec:Bernoulli}

In this section, we carry out some explicit calculations when $ \env=(\ww_i)_{i\in\mathbb{Z}}$ is a sequence of \emph{i.i.d.\@} random variables distributed as a Bernoulli random variable times a constant, \emph{i.e.}, whose distribution is given by 
$$
\nu_{p,M}=p\,\delta_M + (1-p)\,\delta_0,
$$
for some $p\in (0,1)$ and $M>0$. 

Note that Hypothesis \hyperref[hyp:nu]{Hyp-$\nu$} is satisfied. \cref{th:transient} applies, and the corresponding extremal MERW $X^+$ has a positive linear speed denoted here by $v_{p,M}$. 
Our aim is to study the behavior of the speed limit $v_{p,M}$ in the extreme cases $p\to 0$ and $p\to 1$.

We need a few notations. Denote by $\overline{\mathbb{Z}}$ the usual unweighted directed graph on $\mathbb{Z}$, where a loop is added at each vertex (i.e., the adjacency matrix is a $0/1$ matrix). Then, given $n\geq \ell\geq k\geq 0$, let $c_n(k, \ell)$ be the number of excursions of length $n$ from $0$ to $0$ into $\overline{\mathbb{Z}}$ that include exactly $\ell$ loop steps and visit exactly $k$ distinct loops among them.

Obviously, one has $c_n(0,\ell)=0$ for all $n\geq \ell\geq 1$ and, by convention, $c_0(0,0)=1$. To our knowledge, no explicit closed form for $c_n(k, \ell)$ is known for arbitrary $n\geq \ell\geq k\geq 0$, and we have not been able to find one either.

When $\ell=k=0$, $c_n(0,0)$ is nothing but the number of excursions of length $n$ from $0$ to $0$ in the usual lattice $\mathbb{Z}$ (with no loops). Hence, for all $n\geq 0$,
\begin{equation}\label{eq:retours0simple}
c_{2n}(0,0)=\binom{2n}{n}\quad\text{and}\quad c_{2n+1}(0,0)=0.
\end{equation}

\begin{lem}\label{lem:speedseries}
The inverse of limiting speed $v_{p,M}$ satisfies
\begin{equation}\label{eq:vp_series}
\frac{1}{v_{p,M}}=
\sum_{n\geq \ell \geq k\geq 0}c_n(k,\ell) \frac{p^kM^\ell}{(2+M)^n}.
\end{equation}	
\end{lem}

\begin{proof}[Proof of Lemma \ref{lem:speedseries}] First, recall that $\lambda=2+M$ and that the inverse of the speed $v_{p,m}$ is equal to $\mathbb{E}_\mu[S]$ (see \eqref{eq:drift}). In order to prove \eqref{eq:vp_series}, we use the first equality in \eqref{eq:joliS}. It follows that
\begin{equation}\label{eq:v_p_m}
v_{p,M}^{-1}=
\sum_{n\geq 0}{\mathbb E_\mu[a_{0,0}^{(n)}]}\,{\lambda^{-n}}=\sum_{n\geq 0}\sum_{\gamma\in\mathcal E_n^0}\mathbb E_\mu[a_\gamma]\lambda^{-n},
\end{equation}
where $\mathcal E_n^0$ is the set of excursions of length $n$ from $0$ to $0$  
in $\overline{\mathbb Z}$ and $a_\gamma$ is defined as in \eqref{UniformPaths}, that is as $a_{i_0,i_1} \dots  a_{i_{n-1},i_n}$ if  $\gamma = i_0\to \cdots\to i_n$. By decomposing such excursions according to the number of loops $\ell$ and the number of distinct loops $k$, one obtains \eqref{eq:vp_series}. 

Indeed, note that $a_{i,j}=1$ when $i\neq j$ and $a_{i,i}= \ww_i$. So we only need to focus on the loops appearing in the excursion $\gamma$ to compute $\mathbb{E}_\mu[a_\gamma]$. Besides, since the $\ww_i$'s are independent, $\mathbb{E}_\mu[a_\gamma]$ can be written as a product where, if a loop $i \to i$ appears $m$ times in $\gamma$, its contribution to the product is simply $\mathbb{E}_\mu[a_{i,i}^m]=p M^m$. This completes the proof.
\end{proof}

\begin{prop}\label{asympbernoulli} 
The map $p\longmapsto v_{p,M}$ is decreasing and smooth on $(0,1)$ and
\begin{equation}\label{eq:limits_v}
\lim_{p\to 0} v_{p,M}=\sqrt{1-\frac{4}{(2+M)^2}}\quad\mbox{and}\quad v_{p,M}\underset{p\to 1}\sim \frac{3(1-p)}{2+M}.
\end{equation}
\end{prop}


\begin{rem}
If $p=0$, the only MERW $(X_n)_{n\geq 0}$ is the simple symmetric random walk on $\mathbb{Z}$. The latter is recurrent and satisfies $\lim_{n\to+\infty} {X_n}/{n}=0$ by the strong law of large numbers. This shows, in particular, that the limiting speed is discontinuous as $p\to 0$.	
\end{rem}

\begin{rem}
	The first derivative of $v_{p,M}$ at $p=0$ can be obtained in a somewhat tedious manner using standard algebraic combinatorial methods. 
\end{rem}

\begin{proof}[Proof of \cref{asympbernoulli}]

To begin with, the fact that $p\mapsto v_{p,M}$ is a smooth and increasing function comes directly from \eqref{eq:vp_series}. Besides, 
$$
\lim_{p\to 0} v_{p,M}^{-1}
= 
\sum_{n\geq 0} \frac{c_n(0,0)}{(2+M)^n}
\stackrel{\text{\eqref{eq:retours0simple}}}{=} 
\sum_{p\geq 0} \binom{2p}{p}\frac{1}{(2+M)^{2p}}=\left(1-\frac{4}{(2+M)^2}\right)^{-1/2}.
$$
	
We now turn to the most technical part, which is the expansion of $v_{p,M}$ when $p\to 1$. Recall from \cref{lemScombinatoir} (especially from \cref{rem:joliS}) that 
\begin{equation}\label{eq:forMC}
v_{p,M}^{-1}=\mathbb{E}_\mu\left[\frac{\lambda}{\lambda -\ww_0-\beta_{-1}(\env)-\alpha_1(\env)}\right].
\end{equation}

Then, introduce the \emph{i.i.d.\@} geometric random variables of parameter $1-p$ given by 
$$
T^+=\inf\{i\geq 1 : \ww_i=0\}\quad\mbox{and}\quad T^-=\inf\{i\geq 1 : \ww_{-i}=0\}.
$$
We get from \eqref{eq:beta} and \eqref{eq:alpha} that
$$
\beta_{-1}(\env)=g_{M}^{T^--1}(\beta_{-T^-}(\env))\quad\mbox{and}\quad \alpha_{1}(\env)=g_{M}^{T^+-1}(\alpha_{T^+}(\env)).
$$
Recall that the functions $g_s$, for $0\leq s\leq M$, are defined in \eqref{iterated} and represented in Figure \ref{fig:g_0_g_M}. In particular, we recall that $g_0(x)=(\lambda-x)^{-1}$ and $g_M(x)=(2-x)^{-1}$. 

To go further, observe the following equalities in distribution:
$$
\left(\ww_{T^++i}\right)_{i\geq 1}\stackrel{\text{(d)}}{=} \left(\ww_{-T^--i}\right)_{i\geq 1}\left(\stackrel{\text{(d)}}{=} (\nu_{p,M})^{\otimes \bbZ_{\geq 0}}\right).
$$
Clearly, the two latter sequences are  independent of each other. Then, we deduce from \eqref{eq:alpha_fraction_continue} that $\alpha_{T^+}$ and $\beta_{-T^-}$ are \emph{i.i.d.\@} random variables, independent of $T^+$, $T^-$ and $(\ww_{i})_{-T^-\leq i\leq T^+}$, and we obtain that $\alpha_{T^+}$ and $\beta_{-T^-}$ are distributed as $g_0(\alpha_1)=(\lambda-\alpha_{1})^{-1}$. 

Let $Y_1$ and $Y_2$ be two independent random variables distributed as $(\lambda-\alpha_{1})^{-1}$ and independent from $\ww_0$. Then, for any $n,m\geq 0$,
\begin{multline}
\mathbb E_\mu\left[\frac{\lambda}{\lambda -\ww_0-\beta_{-1}-\alpha_1}\bigg|T^+=n,T^-=m\right] =
\mathbb E\left[\frac{\lambda}{\lambda-\ww_0-g_M^{n-1}(Y_1)-g_{M}^{m-1}(Y_2)}\right]\\
= \mathbb (1-p)\mathbb E\left[\frac{\lambda}{\lambda-g_M^{n-1}(Y_1)-g_{M}^{m-1}(Y_2)}\right]+ p \mathbb E\left[\frac{\lambda}{2-g_M^{n-1}(Y_1)-g_{M}^{m-1}(Y_2)}\right].\label{conditionalexpectation}
\end{multline}

By using \eqref{item:encadrement_alpha}, one gets that the expectation in front of $1-p$ in \eqref{conditionalexpectation} is bounded by $\lambda/(\lambda-2)$. Now, we focus on the second term (in front of $p$). To this end, we need the following result. 

\begin{lem}\label{arithmetique}
	Let $(u_n)_{n\geq 0}$ be defined recursively by $u_0<1$ and $u_{n+1}=g_M(u_n)$ for all $n\geq 0$. Then, for every $n\geq 0$, 
	$$
	u_n=1-\frac{1}{\frac{1}{1-u_0}+n}.
	$$
\end{lem}

\begin{proof}[Proof of \cref{arithmetique}]
	It suffices to see that $1/(1-u_n)$ is an arithmetic sequence.
\end{proof}

Thereafter, since $\gamma < \alpha_1 < 1$, one can easily check that $\gamma < Y_1, Y_2 < (\lambda-1)^{-1} < 1$ (we refer to Figure \ref{fig:g_0_g_M}). Hence, the random variables $(1-Y_1)^{-1}$ and $(1-Y_2)^{-1}$ are bounded from above and below, and we obtain from \cref{arithmetique} that, for some positive constants $d,d^\prime$ (depending on $M$), we have for all $n, m \geq 1$,
\begin{equation}\label{bound2}
\frac{\lambda}{\frac{1}{d+n}+\frac{1}{d+m}} \leq \mathbb{E}\left[\frac{\lambda}{2-g_M^{n-1}(Y_1)-g_{M}^{m-1}(Y_2)}\right]\leq 
\frac{\lambda}{\frac{1}{d^\prime+n}+\frac{1}{d^\prime+m}}.
\end{equation}


To conclude, write $T^+=\lfloor T_1\rfloor +1$ and $T^-=\lfloor T_2\rfloor +1$, where $T_1$ and $T_2$ are independent and exponentially distributed random variables with mean $-1/\ln(p)$. 

By noting that $T_1\leq T^+\leq T_1+1$ and $T_2\leq T^-\leq T_2+1$, one can easily see that, for arbitrary $\delta>0$, there exist positive constants $a,b,c,a',b',c'$ such that
\begin{equation}\label{eq:bound3}
\frac{(a+T_1)(b+T_2)}{c+T_1+T_2}\leq \frac{1}{\frac{1}{\delta+T^+}+\frac{1}{\delta+T^-}}\leq \frac{(a'+T_1)(b'+T_2)}{c'+T_1+T_2}.
\end{equation}

\begin{lem}\label{exponential}
	Let $T_1$ and $T_2$ be two independent exponentially distributed random variables with mean $1/r$. For arbitrary positive real numbers $x,y,z$, one has
	$$
	\mathbb{E}\left[\frac{(x+T_1)(y+T_2)}{z+T_1+T_2}\right]\underset{r\to 0}{\sim} \frac{1}{3r}.
	$$
\end{lem}

\begin{proof}[Proof of \cref{exponential}]
	Straightforward estimates show that
	$$
	\mathbb{E}\left[\frac{(x+T_1)(y+T_2)}{z+T_1+T_2}\right]\underset{r\to 0}{\sim}
	\mathbb{E}\left[\frac{T_1T_2}{T_1+T_2}\right]=\frac{1}{3r}.
	$$
\end{proof}

Finally, plugging \cref{exponential}, \eqref{eq:bound3}, and \eqref{bound2} into \eqref{conditionalexpectation} yields
$$
\mathbb{E}_\mu\left[\frac{\lambda}{\lambda -\ww_0-\beta_{-1}-\alpha_1}\right]\underset{p\to 1}{\sim}\frac{\lambda}{-3\ln(p)}\underset{p\to 1}{\sim} \frac{\lambda}{3(1-p)}.
$$
This completes the proof.
\end{proof}

We refer to Figure \ref{vitessep}, which illustrates the variation of the speed with respect to $p$ for different values of $M$. The latter is obtained using Monte Carlo simulations, leveraging the formula \eqref{eq:forMC}.

\begin{figure}[H]
	\centering
	\includegraphics[scale=0.575]{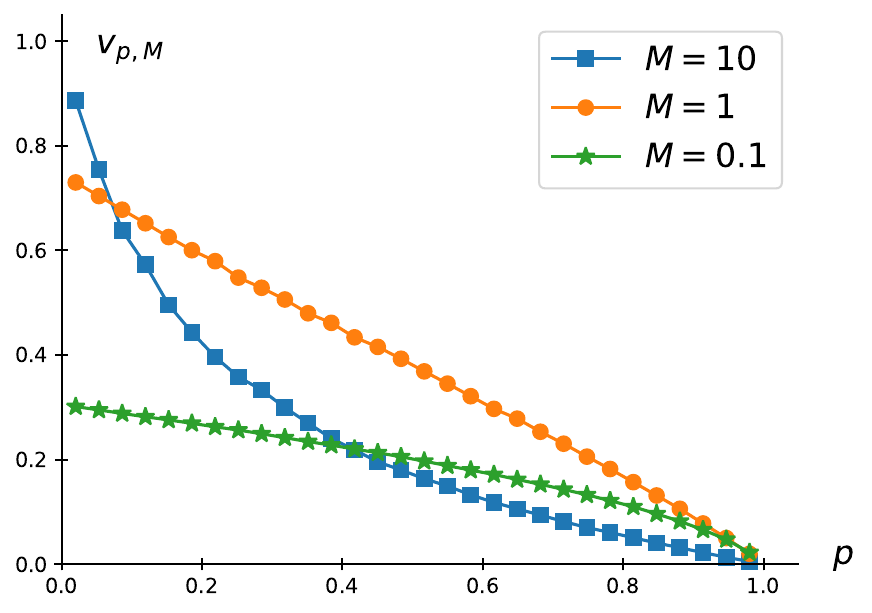}
	\caption{Sketch of plots of $p\longmapsto v_{p,M}$ obtained by Monte Carlo simulations.}	
	\label{vitessep}
\end{figure}

\section{Discussion and remaining questions}

We have seen that MERWs on $\bbZ$ with loops are very rich, and their behaviors are highly sensitive to modifications in the loop environment. Moreover, the methods and results are also very different from those obtained by \cite{Localization} in the finite case $\bbZ/n\bbZ$. 

We hope that our combinatorial approach may lead to new results in the future. The following questions remain to be investigated:
\begin{itemize}
	\item A natural generalization is to assign non-constant weights to edges of the form $(i,i+1)$. We have not carried out all the calculations in this direction, but we believe that our approach generalizes well, at least if the weights are bounded away from $0$ and $1$.
	\item We hope that the combinatorial approach could lead to interesting results for other one-dimensional graphs: $\bbZ\times \{0,1\}$, $\bbZ$ with edges between vertices at distances $\geq 2$, etc. 
	
	Also, the use of generating series could potentially yield relevant results for certain families of infinite trees (random or not) and should be compared to the analytic approaches \cite{ochab2012exact}.
	On the other hand, analyzing MERWs on $\bbZ^d$ ($d\geq 2$) appears to be significantly more difficult. This will clearly require different tools.
	\item We leave open the fine properties of the processes $(\alpha_i)_{i\in\mathbb Z}$ and $(\beta_i)_{i\in\mathbb Z}$, even in the case of a random environment with a very simple common distribution $\nu$. In particular, it would be interesting to investigate the support of the stationary distribution of those stochastic processes, which seems singular with respect to the Lebesgue measure.
	\item In view of Figure \ref{vitessep}, it seems that $p\mapsto v_{p,M}$ is convex for $M\geq M_c$ and concave for $M\leq M_c$ for some critical value $M_c$ close to one. This point deserves to be studied and proven if possible.
	
	\item Finally, as in the toy example at the end of Section \ref{sec:toy}, it could be interesting to study the asymptotic behavior of the drift in \cref{asympbernoulli} with respect to the intensity $M$ of the loops. Some of our techniques can be used to show that $v_{p,M}$ converges to $0$ as $M\to 0$ and $M\to+\infty$, but it could be worthwhile to derive precise asymptotics. 
	
	In addition, $M\longmapsto v_{p,M}$ is not monotone, and the maximum of this function (as well as the point where it is reached) warrants further investigation. We refer to Figure \ref{vitesseM}.
\end{itemize}

	\begin{figure}[H]
	\centering
	\includegraphics[scale=0.5]{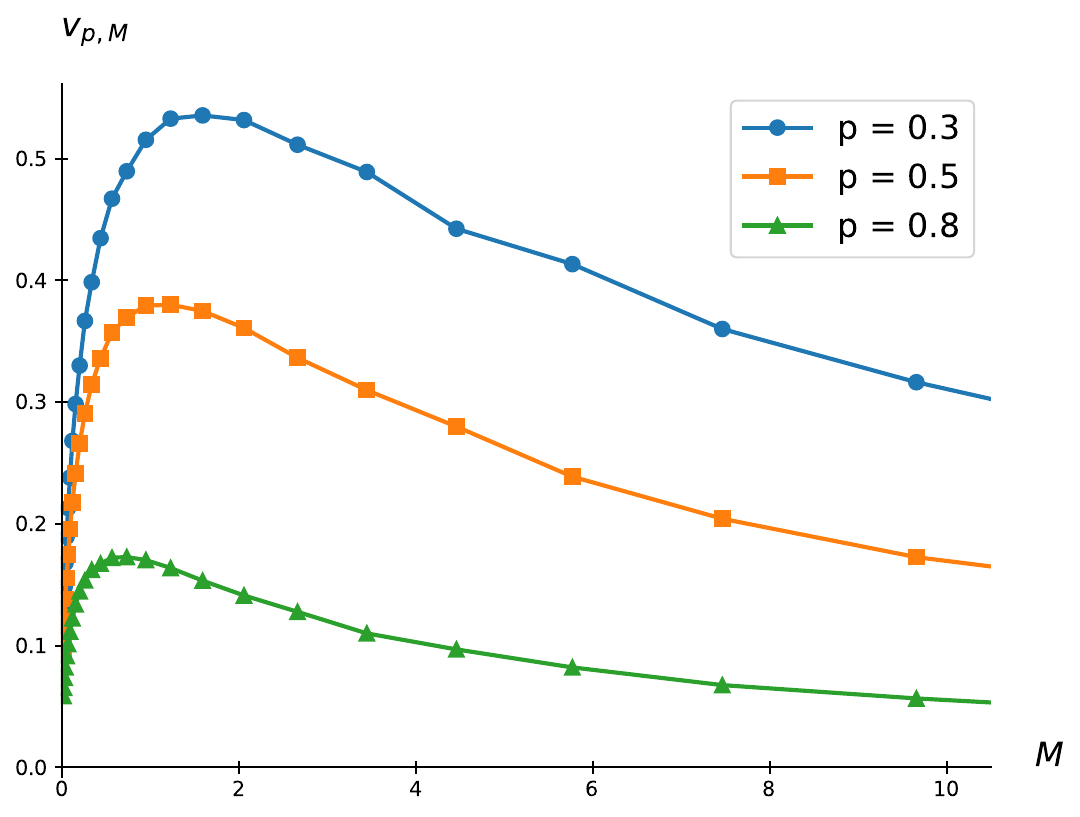}
	\caption{Sketch of plots of $M\longmapsto v_{p,M}$ obtained by Monte Carlo simulations.}		
	\label{vitesseM}
\end{figure}

\noindent
\paragraph{Acknowledgements.} The authors thank the anonymous referees for their careful reading and helpful comments, 
which have contributed to improving the presentation of the paper.

\appendix
\section{Non-random periodic loop environments}

In the case of Bernoulli \emph{i.i.d.\@} loops distributed as $\nu_{p,M}$, \cref{th:transient} combined with \eqref{eq:limits_v} establishes that even for a very small proportion $p$ of loops (each of weight $M$), the extremal MERW $(X_n^+)_{n\geq 0}$ has an asymptotically linear speed that is uniformly bounded away from zero.

To illustrate the specificity of the MERW model in a random environment, we detail in this appendix some computations for a seemingly comparable model: a periodic deterministic loop environment of period $\ell$.

Consider an $\ell$-periodic loop environment, with $\ell\geq 2$, given by $\ww_{n\ell}=M$ for all $n\in\mathbb{Z}$ and $\ww_{i}=0$ for $i\in\mathbb{Z} \setminus \ell\mathbb{Z}$. If we take $p=1/\ell$, this periodic environment has approximately the same number of loops on any large enough interval as the random loop environment $\nu_{p,M}^{\otimes \mathbb{Z}}$.
It is easy to see that the MERW in a periodic environment is recurrent (with zero velocity); this can be proved by looking at the process at the times when it hits a point of $\ell \bbZ$. This appendix describes the process in more detail.

Thanks to graph symmetry, the combinatorial spectral radius can be computed using the tools developed in \cite{Duboux}, as well as $\lambda$-eigenvectors.
Roughly speaking, it suffices to study a reduced graph, which is given here by $\mathbb{Z}/\ell\mathbb{Z}$ with its classical nearest-neighbor structure and, in addition, one loop of weight $M$ at $0$, all the remaining edges being of weight $1$ (see \cref{reduced}).

\begin{figure}[H]
	\centering
	\includegraphics[height=4cm]{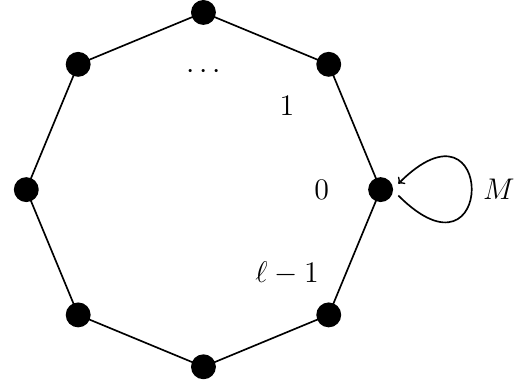}	
	\caption{The reduced graph}
	\label{reduced}
\end{figure}

Let $\lambda_{\ell,M}=2\cosh(\theta_{\ell,M})$ denotes the spectral radius and let $\psi_{\ell,M}$ be the corresponding eigenfunction on $\mathbb Z/\ell \mathbb Z$, extended by periodicity on $\mathbb Z$. Writing the boundary equation at $0$, and using  $\psi_{\ell,M}(1)=\psi_{\ell,M}(-1)$ by symmetry, it follows that $\theta_{\ell,M}$ is the positive solution of  
\begin{equation*}
2\tanh\left(\frac{\ell\theta_{\ell,M}}{2}\right)\sinh(\theta_{\ell,M})=M,
\end{equation*} 
and, for all $0\leq n\leq \ell-1$ ,  
\begin{equation*}
\psi_{\ell,M}(n)=\cosh\left(\left(n-\frac{\ell}{2}\right)\theta_{\ell,M}\right).
\end{equation*}
We denote by $\pi_{\ell,M}$ the reversible measure of the associated MERW. Recall that $\pi_{\ell,M}$ is proportional to $\psi_{\ell,M}^2$ since the reduced graph is symmetric.

Observe that $\theta_{\ell,M}$ (hence $\lambda_{\ell,M}$) decreases with respect to $\ell$ and increases with respect to $M$. Letting $\ell\to\infty$ while keeping $M$ fixed, one can check that  
\begin{equation*}
\theta_M^\ast:=\lim_{\ell\to\infty}\theta_{\ell,M}= \ln\left(\frac{M+\sqrt{M^2+4}}{2}\right) \quad\text{and}\quad
\lambda_M^\ast:=\lim_{\ell\to\infty}\lambda_{\ell,M}=\sqrt{M^2+4}. 
\end{equation*}
Surprisingly, $\lambda^*_M$ is not equal to $2$, which would correspond to the combinatorial spectral radius when $\env\equiv 0$. As a matter of fact, it corresponds to the case where there is one loop of weight $M$ at some point, say $\ww_i=M$, and $\ww_j=0$ for all $j\neq i$.

\begin{prop}\label{locperiodic} 
For any $M,\varepsilon>0$, one has
\begin{equation}\label{concentration}
\liminf_{\ell\to+\infty}\pi_{\ell,M}\left(\left\{n\in \mathbb Z/\ell\mathbb Z: |n|\leq \frac{1}{2\theta_M^\ast}\ln\left(\frac{1}{\lambda_M^*\varepsilon}\right)\right\}\right)\geq 1-\varepsilon.
\end{equation}
\end{prop}

\begin{proof}[Proof of  \cref{locperiodic}]
By using the formulas $\cosh^2(x)=\frac{1}{2}\left(\cosh(2x)+1\right)$ and  $$\sum_{n=0}^d\cosh(an+b)=\frac{\sinh\left((d+1)\frac{a}{2}\right)}{\sinh\left(\frac{a}{2}\right)}\cosh\left(b+d\frac{a}{2}\right),$$
we get that
$$\sum_{n=0}^d \psi_{\ell,M}^2(n)=1+\frac{d}{2}+\frac{1}{2}\frac{\sinh\left((d+1){\theta_{\ell,M}}\right)}{\sinh\left({\theta_{\ell,M}}\right)}\cosh\left((d-\ell){\theta_{\ell,M}}\right).$$
In particular, the normalizing constant $Z_{\ell,M}$ is equal to
$$Z_{\ell,M}=\frac{\ell+1}{2}+\frac{1}{2}\frac{\sinh\left(\ell\,  {\theta_{\ell,M}}\right)}{\sinh\left({\theta_{\ell,M}}\right)}\cosh\left({\theta_{\ell,M}}\right).$$ 
Besides,  as $\ell\to+\infty$, one can check that
$$
\theta_{\ell,M}=\theta_M^\ast+\frac{2M}{\sqrt{M^2+4}}e^{-\ell  \theta_M^*}+o\left(e^{-\ell  \theta_M^*}\right).
$$ 
We deduce that
$$
Z_{\ell,M}\stackrel{\ell\to+\infty}{\sim}\frac{\cosh(\theta^\ast_M)}{4\sinh(\theta^\ast_M)}e^{\ell \theta_M^\ast}  \quad\text{and}\quad
 \sum_{n=-d}^d \psi_{\ell,M}^2(n)\stackrel{\ell\to+\infty}{\sim} \left(\frac{\sinh((d+1)\theta_M^\ast)e^{-d\theta_M^\ast}}{2\sinh(\theta_M^\ast)}-\frac{1}{4}\right) e^{\ell \theta_M^\ast}.
 $$
Since $2\sinh((d+1)x)e^{-dx}-\sinh(x)=\cosh(x)+\frac{1}{2}e^{-2dx}$ one can write 
$$
\frac{\sum_{n=-d}^d \psi_{\ell,M}^2(n)}{Z_{\ell,M}}\stackrel{\ell\to+\infty}{\sim} 1-\frac{e^{-2d\theta_M^\ast}}{2\cosh(\theta^\ast_M)}.
$$
Assuming $d$ is the lowest integer greater than $\frac{1}{2\theta^\ast_M}\ln\left(\frac{1}{ \lambda_M^\ast\varepsilon}\right)$, we obtain the result.
\end{proof}

\begin{rem}
	Regarding the corresponding MERW $(X_n)_{n\geq 0}$ on $\mathbb{Z}$ with the $\ell$-periodic loop environment $\env$, one can easily show that it is null recurrent. Hence, this is the unique MERW. Besides, denoting by $\mathbb{P}_k^{\ell,M}$ its distribution starting from $k\in\mathbb{Z}$, and $d(x,A)$ the distance between $x\in\mathbb{Z}$ and $A\subset \mathbb{Z}$, the concentration inequality \eqref{concentration} can be written as  
	$$\liminf_{\ell\to+\infty}\lim_{n\to\infty}\mathbb{P}_k^{\ell,M}\left(d(X_n,\ell \mathbb{Z})
	\leq \frac{1}{2\theta_M^\ast}\ln\left(\frac{1}{\lambda_M^*\varepsilon}\right)
	\right)\geq 1-\varepsilon.$$ 	
	
\end{rem}

In conclusion, by comparing \eqref{eq:limits_v} and \cref{locperiodic}, we can see that the randomness of loops has a very significant effect. Indeed, even for a tiny $p$, MERW in a random environment escapes at linear speed, whereas MERW in a deterministic periodic environment remains recurrent and highly localized around the loops.

\bibliography{biblio_anderson}
\bibliographystyle{alpha}

\vfill

\end{document}